%% file: arXiv.tex
\documentclass[preprint,12pt]{elsarticle}

\usepackage{amssymb}
\usepackage{amsthm}
\usepackage{amsmath,amsfonts}
\usepackage{graphicx}
\usepackage{geometry}
\usepackage{wrapfig}
\usepackage{textcomp}
\usepackage{cleveref}
\usepackage{subfig}
\usepackage{tabularx}
\usepackage{bm}
\usepackage{caption}
\usepackage{parskip}
\usepackage{url}
\usepackage{tikz}
\usetikzlibrary{shapes.geometric, arrows}
\tikzstyle{startstop} = [rectangle, rounded corners, 
minimum width=3cm, 
minimum height=1cm,
text centered, 
draw=black, 
fill=red!30]

\tikzstyle{io} = [rectangle,  
minimum width=2cm, 
minimum height=1cm, text centered, text width=3cm,
draw=black, fill=blue!30]

\tikzstyle{process} = [rectangle, 
minimum width=3cm, 
minimum height=1cm, 
text centered, 
text width=3cm, 
draw=black, 
fill=orange!30]

\tikzstyle{decision} = [diamond, 
minimum width=3cm, 
minimum height=1cm, 
text centered, 
draw=black, 
fill=green!30]
\tikzstyle{arrow} = [thick,->,>=stealth]

\biboptions{sort&compress}

\usepackage{titlesec}

\usepackage{subfig}

\titlespacing*{\section}
 {0pt}{0.75\baselineskip}{0.75\baselineskip}
\titlespacing*{\subsection}
 {0pt}{0.5\baselineskip}{0.5\baselineskip}
\def\BibTeX{{\rm B\kern-.05em{\sc i\kern-.025em b}\kern-.08em
    T\kern-.1667em\lower.7ex\hbox{E}\kern-.125emX}}
\allowdisplaybreaks
\input{PSmacros}

\newcommand{\ha}[1]{\textcolor{black}{#1}}

\newcommand{\hesham}[1]{\mycomment}{}
\newcommand{\ps}[1]{\textcolor{black}{#1}}
\newcommand{\pgs}[1]{\textcolor{black}{#1}}

\newtheorem{thm}{Theorem}

\newdefinition{rmk}{Remark}
\newproof{pf}{Proof}
\newproof{pot}{Proof of Theorem \ref{thm2}}

\journal{Journal of the Franklin Institute
}

\begin{document}

\begin{frontmatter}



\title{A New Nonsmooth Optimal Control Framework for Wind Turbine Power Systems}

\author[inst1]{Hesham Abdelfattah}
\ead{abdelfhm@mail.uc.edu}
\author[inst1]{*Sameh A. Eisa}
\ead{eisash@ucmail.uc.edu}
\affiliation[inst1]{organization={Department of Aerospace Engineering and Engineering Mechanics, University of Cincinnati},
            addressline={2600 Clifton Ave}, 
            city={Cincinnati},
            postcode={45221}, 
            state={OH},
            country={USA}}

\author[inst2]{Peter Stechlinski \corref{cor1}}
\ead{peter.stechlinski@maine.edu}
\affiliation[inst2]{organization={Department of Mathematics and Statistics, University of Maine},
            city={Orono},
            postcode={04469}, 
            state={Maine},
            country={USA}}
\tnotetext[t1]{This material is based upon work supported by the National Science Foundation under Award No. 2318772 and 2318773.
}

\begin{abstract}
Optimal control theory extending from the calculus of variations has not been used to study the wind turbine power system (WTPS) control problem, which aims at achieving two targets: (i) maximizing power generation in lower wind speed conditions; and (ii) maintaining the output power at the rated level in high wind speed conditions. A lack of an \textit{optimal control} framework for the WTPS (i.e., no access to actual optimal control trajectories) reduces optimal control design potential and prevents competing control methods of WTPSs to have a reference control solution for comparison. In fact, the WTPS control literature often relies on reduced and linearized models of WTPSs, and avoids the nonsmoothness present in the system during transitions between different conditions of operation. In this paper, we introduce a novel optimal control framework for the WTPS control problem. We use in our formulation a recent accurate, nonlinear differential-algebraic equation (DAE) model of WTPSs, which we then generalize over all wind speed ranges using nonsmooth functions. We also use developments in nonsmooth optimal control theory to take into account nonsmoothness present in the system. We implement this new WTPS optimal control approach to solve the problem numerically, including  (i) different wind speed profiles for testing the system response; (ii) real-world wind data; and (iii) a comparison with  smoothing and naive approaches. Results show the effectiveness of the proposed approach. 

\end{abstract}




\end{frontmatter}


\section{Introduction}\label{section:I}\label{sec:introduction}

According to the U.S.\ Department of Energy, as can be seen in this report \cite{wiser2022land}, wind energy is recognized as a prominent contributor to electric-power capacity growth in the U.S. Consequently, considerable efforts have been dedicated in the past two decades to enhancing the performance of wind turbine power systems (WTPSs) and integrating them with traditional power systems and the power grid. However, given their distinct nature compared to  traditional power systems, integrating WTPSs with the existing power grid remains a significant challenge \cite{liang2016emerging} that is being  studied by national labs, industry and by academics
\cite{Pourbeik2014,GE2010,eisa2019modeling,eisa2019nonlinear}.
Challenges in WTPSs include inherent fluctuations, limited predictability, reduced grid stabilization inertia, and heightened sensitivity to uncontrolled factors such as weather conditions.
Consequently, research endeavors encompassing control systems, optimization, and power generation/storage have significantly surged over the past two decades \cite{eisa2019modeling,eisa2019nonlinear,kusiak2010power,lu2009combining,eisa2020investigating,menezes2018review,Tsourakis2009,gao2016pitch,lan2018fault,zhaocombined,kumar2016review,abdelfattah2023novel,hansen2024condition,golnary2019dynamic,kear2016computational}, where major efforts have been made to have a better understanding of the response of the WTPS to sudden changes of uncontrolled parameters, such as the wind speed, power grid loads and the bus frequency or voltage \cite{GE2010,eisa2019nonlinear,eisa2019modeling,Tsourakis2009,eisa2020investigating,Rahimi2014,Eisa1611:Sensitivity,ghosh2015doubly}. 

\subsection{Modeling and Control of WTPSs}\label{subsection:I.A}

Traditionally, there are two control inputs/strategies for WTPSs \cite{eisa2019modeling,menezes2018review}: (i) pitch angle control, and (ii) torque/tip-ratio control. 
Using the first strategy (i.e., pitch control), we can control the power produced from wind by changing the pitch angle \cite{eisa2019modeling,eisa2019nonlinear,GE2010}. In the second strategy, the torque/tip-ratio based control methods enable changing the rotor speed of the WTPS to maximize power production; a common strategy/method for torque control is known as maximum power point tracking (MPPT) where the objective is to maximize power production through changing the speed of the turbine 
\cite{menezes2018review,kumar2016review,6074308,abdullah2012review}.
In this paper, we focus mainly on pitch angle control.
Pitch control allows one to change the pitch angle in order to control the aerodynamic loads on the rotor. This affects the power coefficient, known as $C_p$, of the wind turbine, which in turn changes the amount of power produced from the WTPS \cite{eisa2019modeling,eisa2019nonlinear,GE2010}.
Active pitch control methods are used \cite{eisa2019modeling,eisa2019nonlinear,GE2010,menezes2018review} to control the pitch angle typically in Region 3, where the generator torque is usually kept constant and the pitch angle is varied to keep the generated electric power at the rated level.
Pitch control has been traditionally implemented and designed 
 using proportional-integral-derivative (PID) controllers \cite{eisa2019modeling,eisa2019nonlinear,Tsourakis2009,GE2010,menezes2018review}.
 There are some studies aimed at improving the efficiency of the PID controllers, such as using estimation and optimization techniques \cite{gao2016pitch} and  genetic algorithm techniques  \cite{civelek2016control}.
 Additionally, feedforward, Model predictive control (MPC), sliding mode and other controllers have been also considered (e.g., \cite{koerber2013combined,zhaocombined,chang2017fuzzy,golnary2019dynamic}).
It is worth noting that some works in literature have considered both pitch and torque/tip-ratio controls simultaneously such as \cite{soliman2011multiple,Jafarnejadsani2015}.

\textcolor{black}{In WTPSs control literature, many studies use linearized/reduced and/or models in the frequency/s-domain (e.g., \cite{slootweg2003general,hiskens2011dynamics,Pourbeik2014,GE2010,Lasheen2016,hu2022novel,li2020dynamic}) while testing and analyzing newly developed control methods. Recently, significant progress has been made in WTPS modeling \cite{eisa2019modeling,eisa2019nonlinear}. In said recent works, WTPSs have been rigorously studied and converted into a full scale nonlinear time-domain differential-algebraic-equation (DAE) model based on physical/industrial literature such as, but not limited to \cite{GE2010,Pourbeik2014,Tsourakis2009}.}
This nonlinear time domain WTPS DAE model \cite{eisa2019modeling,eisa2019nonlinear} will be used in the control study of this paper (in Section \ref{section:III}).
\subsection{
Development\ha{s} in Nonsmooth ODE/DAE Optimal Control}\label{subsection:IB}  

Traditional optimal control methods require differentiability of the system's dynamics as well as the objective functional. 
However, many applications, such as the one this paper is concerned about, i.e.,  WTPSs, exhibit nonsmoothness in the dynamics and the objective functional 
due to transitions between different Regions of operation (see Figure \ref{PmechWref} in Section \ref{section:III}); in fact, a detailed study on nonsmoothness appearing in the WTPS parameters can be found in \cite{eisa2021sensitivity}. Therefore, an extension of traditional optimal control methods is desired to overcome this issue (i.e., study nonsmooth problems).
Two popular methods in the literature used to resolve the nonsmoothness issue include using slack variables and smoothing approximations \ha{\cite{Kelly2017,skandari2016smoothing,caspari2020dynamic,bagirov2013hyperbolic,amhraoui2022smoothing}}. Both the slack variables and the smoothing approximation methods have limitations. \ha{Using the slack variables method 
increases the size of the problem which results in a slower convergence and higher complexity especially for large systems}, while using a smoothing approximation fundamentally changes the optimization problem with the accuracy of approximation depending on the choice of the smoothing parameters which add an element of arbitrariness; in fact, this has been observed directly with relevance to the WTPS \cite{eisa2021sensitivity}. 

Unlike slack variables or smoothing approaches, frameworks now exist that enable the direct treatment of nonsmoothness. For example, there are methods that extend standard indirect methods \ha{(in which calculus of variations is used to furnish a boundary value problem that can be solved
to obtain a trajectory that minimizes/maximizes a given objective functional \cite{pontryagin1987mathematical,Kelly2017})} for optimal control to nonsmooth problems, including constrained nonsmooth dynamical systems (see Ch.5 in \cite{Clarkebook} and the work in \cite{clarke2010optimal}). However, these results generally come with restrictive conditions, and so are practically implementable typically only in special cases, such as linear complementarity systems \cite{guo2016necessary}.
As for nonsmooth direct optimal control methods, lexicographic directional derivative (LD-derivative) \cite{Khan2015_automaticdiff} based nonsmooth optimal control methods for ODE systems and DAE systems were introduced recently in \cite{khan2014generalizedoptimal} and \cite{2016_Stechlinski_Barton_CDC}, respectively.
The advantage of such methods is that LD-derivatives satisfy sharp calculus rules, which makes them a practical tool for calculating computationally relevant  generalized gradient elements of the objective functional. 
LD-derivatives based nonsmooth optimal control methods will be presented in Section \ref{section:II} 
as we will use it in the optimal control study of this paper.
\ha{Moreover, nonsmooth  optimization methods \cite{bagirov2020numerical} can utilize  LD-derivatives based methods.}
\ha{Some studies have applied LD-derivative based methods to solve nonsmooth optimal control problems involving nonsmooth systems with constraints, including nonsmooth DAE systems in \cite{2016_Stechlinski_Barton_CDC}, mixed nonlinear complementary systems\cite{stechlinski2022dynamic}, and DAEs with optimality criteria \cite{stechlinski2020optimization}.}

\subsection{Motivation and Contributions}\label{subsection:IC}  
Despite the significant progress made in WTPS control strategies (as discussed in Section \ref{subsection:I.A}), current research efforts often suffer from three major issues. 
First, model linearization and reduction that can lead to inaccurate conclusions about the control method or its performance (see relevant discussions on this in \cite{abdelfattah2023novel,eisa2019modeling,IJDY17}). 
Second, it is quite common for WTPSs to experience nonsmooth (non-differentiable) profiles in their dynamics, parameters and/or the objective functional (see for example \cite{eisa2021sensitivity}); however, to the authors' best knowledge, no work has been done that directly addresses the nonsmoothness present in the WTPS model for control studies. 
Overcoming nonsmoothness by smoothing has been shown to be less effective in studying WTPSs as shown in \cite{eisa2021sensitivity}. 
Third, there is a lack of optimal control methods extending from the calculus of variations in the sense of \cite{pontryagin1987mathematical,Kelly2017,Rao2010} (such as direct methods)  that find the optimal feasible trajectory for use as a reference control solution/design tool, which can be used by different control methods for realization and/or comparison efforts. 
We are motivated to provide a new \textit{theory and framework} to address these issues. To overcome the first issue, the recent nonlinear time domain WTPS DAE model introduced in \cite{eisa2019modeling,eisa2019nonlinear} \ha{(see relevant discussion in subsection \ref{subsection:I.A})} will be used for \ha{the following reasons: (i) to the best of the authors' knowledge, it is the only WTPS model that has been shown to be mathematically well-posed (the existence and uniqueness properties have been proved in \cite{IJDY17}); (ii) this model is shown to be amenable to available, non-commercial, stiff numerical solvers of DAE systems \cite{eisa2019modeling}; (iii) it has been shown in \cite{eisa2018wind} that reductions based on linear/multiple-scale approximations can be made, however, the separation between fast and slow variables is quite complicated and is not easy to construct, suggesting against reduction/linearization; and (iv) it is accurate when validated against real-world measured data (see \cite{eisa2019modeling,eisa2019nonlinear,eisa2018wind}) 
and outperformed some of the most cited models that existed in the literature as asserted in \cite{eisa2019modeling,eisa2019nonlinear,eisa2018wind}.}
The second and third issues can be addressed using the developments in nonsmooth optimal control theory for dynamical systems with DAE dynamics \cite{2016_Stechlinski_Barton_CDC} (as discussed in Section \ref{section:II}). These developments with relevance to the work of this paper will be re-introduced in Section \ref{section:II}. A simpler case study for a reduced version of the mentioned WTPS model \cite{eisa2019modeling,eisa2019nonlinear} in a particular wind speed range and smooth setting has been conducted successfully by the authors in \cite{abdelfattah2023novel}, which further suggests that the work of this paper is with significant potential and is highly needed.

\textcolor{black}{In this paper, we make the following contributions:}
\begin{enumerate}
    \item \textcolor{black}{We unify and generalize the modeling and control framework for WTPSs which includes/covers all wind speed ranges: low, moderate and high (Regions 1-3, see Figure \ref{PmechWref}). This generalized framework is achieved by making use of nonsmooth functions which enable continuous transition between all different operational conditions (i.e., Regions 1-3) of the WTPS. Our generalized framework is based on the recently developed time-domain, nonlinear, validated against real-world data, DAE WTPS models in \cite{eisa2019modeling,eisa2019nonlinear}. We did not perform any linearization/reduction or smoothing approximation to have the generalized framework as accurate, and reflective to WTPS dynamics, as possible. 
    }
    \item  \textcolor{black}{We provide a novel optimal control method that is suitable for (and specialized to) this \textit{nonsmooth} WTPS generalized modeling and framework using tools from generalized derivatives theory. We thus provide, for the first time, a mathematically rigorous framework (extended from the calculus of variations) that aims to find the feasible optimal trajectory for the pitch angle controller that achieves the WTPS objectives in all Regions of wind conditions simultaneously (i.e., in Regions 1-3, see Figure \ref{PmechWref}): (i) maximizing the power generation (in lower wind speed conditions in Regions 1 and 2); and (ii) keeping the output power at the rated level (in high wind speed conditions in Region 3). 
    For that, we use the established results in \cite{2016_Stechlinski_Barton_CDC} to formulate a nonsmooth sequential optimal control (i.e., a direct method) approach for the WTPS.
 Furthermore, we test our approach and present the optimal control response of WTPSs when exposed to wind variations with different wind profiles (Gaussian and ramp profiles). We also apply our approach to, and present results of, the highly turbulent, real-world wind data that was used in modeling validation in \cite{eisa2019modeling,eisa2019nonlinear,eisa2018wind}.}
\item \textcolor{black}{We demonstrate the importance of this novel approach, which treats the nonsmoothness directly and is based on state-of-the-art accurate models of WTPSs via: (i) comparison with a smoothing approach/approximation (which is shown to introduce error); and (ii) comparison with commercial smooth optimal control solver, namely GPOPS-II \cite{gpops2}, where the nonsmoothness is ignored (which is shown to fail sometimes).}
\end{enumerate}

\textcolor{black}{Given all the contributions listed above, we believe that the optimal control response our generalized framework produces should provide a benchmark (reference for comparison) to assess the performance of any control strategy applied to WTPSs.} 
\textcolor{black}{\subsection{Organization of the Paper}}
\textcolor{black}{In Section \ref{section:II}, we provide a background and mostly self-contained material on solving nonsmooth optimal control via lexicographic differentiation. What is provided in Section \ref{section:II} is customized to the class of DAE systems this paper is concerned with. Then, in Section \ref{section:III}, we provide our first contribution (unification and generalization of WTPS DAE models using nonsmooth functions). Following that, in Section \ref{section:IV}, we provide our second contribution (novel nonsmooth optimal control approach). In Section \ref{section:V}, we provide our third contribution, as we compare the novel nonsmooth  approach derived and tested in Section \ref{section:IV} with a smoothing approximation and a smooth optimal control solver, to demonstrate the effectiveness of our new approach. We conclude the paper in Section \ref{section:VI}.}   
 
\section{{Solving Nonsmooth Optimal Control Problems via \ps{LD-Derivatives}}}\label{section:II}
\pgs{As discussed in \cite{eisa2019modeling,eisa2021sensitivity}, a WTPS can be formulated as a  DAE system in the semi-explicit form. Therefore, we turn to the work in \cite{2016_Stechlinski_Barton_CDC}, which provides the generalized gradient elements of the  objective functional for nonsmooth optimal control problems with (possibly nonsmooth) DAE systems embedded.} 
In particular, the DAE optimal control problem of interest is as follows: 
\begin{subequations}\label{NonSmooth_Optimal_eq:1}
 \begin{align}
       \min\limits_{\uu} \quad \phi(\uu)&={\phi}(t_f,\uu(t_f),\x(t_f),\y(t_f)),\label{NonSmooth_DAEobj}\\
     \text{s.t.}  \quad \dot{\x}(t) &= \h(t,\uu(t),\x(t),\y(t)), \quad \x(t_0)=\x_0,\label{NonSmooth_DAE1}\\
     \bm{0}_{n_y} &= \g(t,\x(t),\y(t)),\label{NonSmooth_DAE2} 
    \end{align}
\end{subequations}
with differential state variables $\x(t) \in D_x \subseteq \real^{n_x}$ (with initial condition $\x_0 \in D_x$),  algebraic state variables $\y(t) \in D_y \subseteq \real^{n_y}$, control $\uu(t) \in D_u \subseteq \real^{n_u}$, and finite time interval $[t_0,t_f]\subset D_t \subseteq \real$, and RHS functions $\h:D_t \times D_u \times D_x \times D_y \to \real^{n_x}$ and $\g:D_t \times D_x \times D_y \to \real^{n_y}$, where   $D_t, D_u, D_x,D_y$  are open and connected sets.
In this paper we use the sequential method (also known as direct single shooting) to solve the optimal control problem in \eqref{NonSmooth_Optimal_eq:1}; in the sequential method, the  continuous-time problem is converted into an nonlinear programming (NLP) by discretizing the control input to produce a finite set of decision variables --- the trajectory of the embedded dynamical system (in this case, the DAE system) is approximated via numerical simulation, and so the decision  variables in the NLP come from the parameterization of the control \cite{Kelly2017,Rao2010}. 
Namely, the optimal control problem in \eqref{NonSmooth_Optimal_eq:1} seeks to minimize the objective functional over some function space, i.e., $\min_{\uu \in \mathcal{U}} \phi(\uu)$  (where, e.g., $\mathcal{U}$ could be the space of Lebesgue integrable functions on some finite time horizon). Instead, we convert this problem into a finite-dimensional problem by parametrically discretizing the control into a finite truncation of basis functions, using the control parameters $\pp \in D_p \subseteq \real^{n_p}$, and then we instead solve $\min_{\pp \in D_p} \phi(\pp)$. In this paper, we focus on a piecewise constant control discretization on $n_s \in \posint$ subintervals $[\tau_0,\tau_1], \ldots, [\tau_{n_s-1},\tau_{n_s}]$:
\begin{align}\label{control_param_full} 
  \uu(t)
  &=\textbf{u}_{(i)} \quad \text{ if } t \in (\tau_{i-1},\tau_{i}],
\end{align} 
with pre-fixed times $\tau_i$ satisfying $t_0 = \tau_0< \tau_1 <\dots< \tau_{n_s-1}< \tau_{n_s}=t_f$, with $\tau_i-\tau_{i-1}=(t_f-t_0)/n_s$, and with control parameters $\pp=(\bm{u}_{(1)},\dots,\bm{u}_{(n_s)}) \in D_p=D_u^{n_s} \subseteq \real^{n_u n_s}$ (i.e., $\bm{u}_{(i)} \in D_u$ for each $i$, and $n_p=n_u n_s$ here). The sequential approach uses the gradient information $\nabla \phi(\pp)$ to update the control parameters in the spirit of $\pp_{k+1} = \pp_{k} -\gamma \nabla \phi(\pp_{k})$ for some step size $\gamma$ and repeats the process until a termination criteria is satisfied (so that $\nabla \phi(\pp_{k}) \approx \zero$). Moreover, the gradient information is obtained using the sensitivity functions $\frac{\partial \x}{\partial \pp}$ and $\frac{\partial \y}{\partial \pp}$ associated with the embedded DAE system via  \begin{equation}\label{eq.nabla}
\nabla \phi(\pp)=\frac{\partial \phi}{\partial \pp}+\frac{\partial \phi}{\partial \x}\frac{\partial \x}{\partial \pp}+\frac{\partial \phi}{\partial  \y }\frac{\partial \y}{\partial \pp}.
\end{equation}
However, motivated by WTPSs, we are interested in the case where the maps $\phi,\h,\g$ are assumed to be nonsmooth and thus not necessarily differentiable. Consequently, the solutions $\x$ and $\y$ of the DAE system may be nonsmooth with respect to $\uu$, invalidating \eqref{eq.nabla}. To overcome this issue, we turn our attention to generalized derivatives theory and lexicographic differentiation to compute (generalized) derivative information of $\phi$ for use in a nonsmooth sequential direct method to solve \eqref{NonSmooth_Optimal_eq:1}.

\subsection{Lexicographic Directional Derivatives}\label{subsection:IIA}
Let $X\subseteq \real^n$ be open and $\f:X\to\real^m$ be locally Lipschitz on $X$. Clarke's generalized derivative  \cite{Clarkebook} of $\f$ at $\xnot \in X$ is given by the convex hull of the Bouligand subdifferential  of $\f$, i.e.,
\begin{equation}\label{eq:Clarke}
\partial_{\rm C}\f(\xnot):=\conv \; \partial_{\rm B}\f(\xnot),
\end{equation}
 where the Bouligand subdifferential is defined as:
\begin{equation*}  
\partial_{\text{B}} \f(\x_0)
:=\left\{\F \in \real^{m \times n}: \exists \x_i \to \x_0 \text{ s.t. } \x_i \in X \setminus Z_{\f} \; \forall i \in \posint \text{ and } \J \f(\x_{i}) \to \F \right\},
\end{equation*}   
where $Z_{\f}$ is the zero measure subset on which $\f$ is not differentiable and $\J\f(\x_i)$ is the Jacobian matrix of $\f$ at $\x_i$. 
Note that in case $\f$ is $C^1$ at $\xnot$, then $\partial_{\rm B} \f(\xnot) = \partial_{\rm C} \f(\xnot)=\{\Jf(\xnot)\}$.
Although elements of Clarke's generalized derivative are useful for nonsmooth numerical methods, it does not satisfy sharp calculus rules (e.g. sum rule --- see \cite{ref:OMS_generalizedderivatives}), which makes it difficult to calculate said elements of Clarke's generalized derivative for complicated functions.
The lexicographic directional derivative (LD-derivative) \cite{Khan2015_automaticdiff}, which is based on lexicographic differentiation \cite{Nesterov2005}, was proposed to overcome this limitation (and other limitations \cite{ref:OMS_generalizedderivatives}), 
which is applicable to lexicographically smooth (L-smooth) functions \cite{Nesterov2005}.
The class of L-smooth functions are locally Lipschitz and directionally differentiable to arbitrary order, which includes all $C^1$ functions, $PC^1$ functions (hence abs-value, min, max), and convex functions (hence  any $p$-norm), as well as the compositions of these functions.  
Following from the definition of L-smooth functions, Nesterov \cite{Nesterov2005} defined the lexicographic derivative (L-derivative) 
for L-smooth functions which
can be seen as a Jacobian-like object, as discussed in
\cite{Nesterov2005, khan2014generalized, Khan2015_automaticdiff}, and is just as useful in nonsmooth numerical methods as Clarke generalized derivative elements.

\pgs{An LD-derivative, denoted $\f'(\xnot;\M)$, can be used to calculate an L-derivative of an L-smooth function $\f$ if the so-called ``directions matrix'' $\M=[\bm{m_1} \quad \bm{m_2} \quad \dots  \quad \bm{m_k}]\in \real^{n \times k}$ has full row rank. Roughly, the columns in  $\M$, which span the domain space if $\M$ has full row rank, act as probing directions. 
The main advantage of the LD-derivative formulation is that it satisfies sharp calculus rules and that  there are closed-form expressions available for nonsmooth elemental functions such as $\max$, $\min$, absolute-value, etc. (see \cite{Khan2015_automaticdiff,ref:OMS_generalizedderivatives}\ps{)}:
\begin{itemize}
\item smooth rule: $\f \in C^1 \Rightarrow\f'(\xnot;\M)=\J\f(\xnot)\M$, e.g.\ $\sin'(x_0;\M)=\cos(x_0)\M$.
\item  abs rule: $\mathrm{abs}'(x_0;\M)=\mathrm{fsign}(x_0,\M)\M$ where the first-sign function $\text{fsign}(\cdot)$ returns the sign of the first nonzero element, or zero if its input is zero.
\item  min rule: for $\M_1, \M_2 \in \real^{1 \times k}$,
\begin{align*}
\min{'}(x_0,y_0;(\M_1,\M_2))
&={\rm \bf slmin}([x_0 \quad \M_1],[y_0 \quad \M_2])\\
&=\begin{cases}
\M_1 &\text{if } \mathrm{fsign}(\begin{bsmallmatrix} x_0 \\ \M_1^{\rm T} \end{bsmallmatrix}-\begin{bsmallmatrix} y_0 \\ \M_2^{\rm T} \end{bsmallmatrix}) \leq 0,\\
\M_2 &\text{if } \mathrm{fsign}(\begin{bsmallmatrix} x_0 \\ \M_1^{\rm T} \end{bsmallmatrix}-\begin{bsmallmatrix} y_0 \\ \M_2^{\rm T} \end{bsmallmatrix}) > 0,
\end{cases}
\end{align*}
i.e., the shifted-lexicographic-minimum ${\bf slmin}$ returns the lexicographic minimum of the two vector arguments, left-shifted by one element. 
\item  componentwise rule: $\f'(\xnot;\M)=(f_1'(\xnot;\M),f_2'(\xnot;\M),\ldots,f_n'(\xnot;\M))$.
\item  chain rule: $[\g\circ \f]'(\xnot;\bm{\mathrm{M}})=\g'(\f(\xnot);\f'(\xnot;\bm{\mathrm{M}})).$
\item sum rule: $[\f+\g](\xnot;\bm{\mathrm{M}})=\f'(\xnot;\bm{\mathrm{M}})+\g'(\xnot;\bm{\mathrm{M}}).$
\item product rule: $[fg](\xnot;\bm{\mathrm{M}})=f'(\xnot;\bm{\mathrm{M}})g(\xnot) +f(\xnot) g'(\xnot;\bm{\mathrm{M}}).$
\end{itemize}
}

Because of these sharp calculus rules, the LD-derivative can be computed accurately, hence, making it a computationally-relevant tool that can be effectively used to furnish an L-derivative from solving a system of linear equations. Namely,
\begin{equation}\label{eq.LD.L.derivative}
\f'(\x_0;\M)=\J_{\rm L}\f(\x_0;\M) \M,
\end{equation}
where $\J_{\rm L}\f(\x_0;\M)$ is the L-derivative of $\f$ at $\x_0$ in the directions $\M$, a computationally relevant generalized derivative object.
For more details on LD-derivative theory and some of its applications, the reader may refer to \cite{Khan2015_automaticdiff,ref:OMS_generalizedderivatives}.


\subsection{Generalized Gradients  of Control Problems with Nonsmooth DAEs  Embedded}\label{subsection:IIA}

\pgs{Returning to the optimal control problem in \eqref{NonSmooth_Optimal_eq:1}, we recall some results from the literature that enable the use of LD-derivatives in nonsmooth optimal control settings: For nonsmooth optimal control with ODE systems, the framework presented in \cite{khan2014generalizedoptimal} provides the generalized gradient elements of the objective functional with nonsmoothness present in the objective functional and/or the system dynamics, where such description of generalized gradient elements of the objective functional is based on the sensitivity analysis results in \cite{khan2014generalized} concerning generalized derivatives of ODE solutions  with  respect  to  the  parameters of the system. The authors in \cite{2016_Stechlinski_Barton_LD_DAEs} established a similar sensitivity analysis theory for semi-explicit DAEs with ``generalized differentiation index one'', which can be exploited to furnish generalized gradient elements of the objective functional of an optimal control problem with an embedded DAE \cite{2016_Stechlinski_Barton_CDC}, in the same spirit as \cite{khan2014generalizedoptimal}. In particular, a generalized gradient element of the objective function in \eqref{NonSmooth_Optimal_eq:1} can be obtained via the following result, which specializes Theorem 3.2 in \cite{2016_Stechlinski_Barton_CDC}.}  
 
  
\begin{thm}\label{thm:cdc}
Suppose that the following conditions hold:
\begin{itemize}[(i)]
\item  Assume that  $\phi,\h,\g$ are  L-smooth functions. 
\item Assume that $\uu$ is parameterized according to \eqref{control_param_full}.
\item Assume that there exists a  solution $(\Tilde{\x},\Tilde{\y})$ of \eqref{NonSmooth_DAE1}-\eqref{NonSmooth_DAE2} on $\left[t_0,t_f\right]$ through the initial data $\{(t_0,\pp_0,\x_0,\y_0)\}$, for some  $\pp_0 \in D_p=D_u^{n_s}$, such that
\begin{equation}\label{eq.clarkeregularity}
\pi_{\y}\partial \g(t,\Tilde{\x}(t),\Tilde{\y}(t))=\{\textbf{G}_{\y}\in\real^{n_y \times n_y}: \exists [\textbf{G}_{t} \; \textbf{G}_{\x} \; \textbf{G}_{\y}] \in \partial \g(t,\Tilde{\x}(t),\Tilde{\y}(t))\}
\end{equation}
contains no singular matrices for any $t\in[t_0,t_f]$.  
\end{itemize}
\pgs{ 
Then a generalized gradient of $\phi$ in \eqref{NonSmooth_Optimal_eq:1} can be calculated according to 
\begin{equation}\label{eq.OCP.gengradient}
\bm{\mu}=\phi'(t_f,\bm{u}_{(n_s)},\Tilde{\x}(t_f),\Tilde{\y}(t_f);\mathbf{0}_{1 \times {n_u n_s}},\coord_{(n_s)}^{\rm T} \otimes \II_{n_u},\Tilde{\textup{\textbf{X}}}(t_f),\Tilde{\textup{\textbf{Y}}}(t_f))^{\rm T} \in \partial \phi(\pp_0),
\end{equation} 
where  the LD-derivative sensitivities $(\Tilde{\textbf{X}},\Tilde{\textbf{Y}})$ solve  the following for $i=1,\ldots,n_s$:
    \begin{align}
         \dot{\textbf{X}}(t) &= \h'(t,\bm{u}_{(i)},\Tilde{\x}(t),\Tilde{\y}(t);(\zero_{1 \times n_s n_u},\coord_{(i)}^{\rm T} \otimes \II_{n_u},\textbf{X}(t),\textbf{Y}(t))), \quad \forall t\in(\tau_{i-1},\tau_{i}], \notag\\
         \bm{0}_{n_y\times n_s n_u} &= \g'(t,\Tilde{\x}(t),\Tilde{\y}(t);(\zero_{1 \times n_s n_u},\textbf{X}(t),\textbf{Y}(t))), \quad t\in[t_0,t_f], \notag\\
        {\textbf{X}}(t_0) &= \zero_{n_x \times n_s n_u}, \label{NonSmooth_Optimal_eq:4}\\
        \zero_{n_y \times n_s n_u}&=\g'(t_0,\x_0,\y_0;(\zero_{1 \times n_s n_u},\zero_{n_x \times n_s n_u},\capy(t_0))), \notag
    \end{align}
where $\otimes$ is the Kronecker product and $\coord_{(i)}\in \real^{n_s}$ denotes the $i^{\rm th}$ standard basis vector.}
\end{thm}

\begin{proof}
The DAE system in \eqref{NonSmooth_DAE1}-\eqref{NonSmooth_DAE2}, with controls $\uu$ parameterized according to \eqref{control_param_full}, 
is in the form of the embedded DAE system in Eq.\ (6) in \cite{2016_Stechlinski_Barton_CDC}, with $\f$ replaced by $\h$, $\f_0(\pp)=\x_0$ a constant, the control $\vv$ absent here, and extra problem parameters absent here. Hence, the L-smoothness of $\phi, \h, \g$ here  guarantee Assumption (i)-(ii) and (v) of Assumption 3.1 in \cite{2016_Stechlinski_Barton_CDC} hold. Moreover, we can rewrite $\uu$ in
\eqref{control_param_full} as
$$\uu(t)=\uu(t;\pp)=\sum_{i=1}^{n_s} \uu_{(i)} \mathcal{X}_{(\tau_{i-1},\tau_i]}(t),$$
i.e., $\Psi_i(t)=\mathcal{X}_{(\tau_{i-1},\tau_i]}(t)$ is the indicator function,  so that Assumption (iii) of Assumption 3.1 in \cite{2016_Stechlinski_Barton_CDC} holds. (Assumption (iv) holds vacuously.) 
The assumptions on the RHS functions $\h, \g$ and the piecewise constant control $\uu$ imply that the solution $(\Tilde{\x},\Tilde{\y})$ on $[t_0,t_f]$ is unique, and the assumption that 
\eqref{eq.clarkeregularity} contains no singular matrices guarantees that said solution is regular in the terminology of  \cite{2016_Stechlinski_Barton_CDC}. 
Hence, Theorem 3.2 in \cite{2016_Stechlinski_Barton_CDC} is applicable, and the conclusions in this theorem follow immediately from it, and Example 3.4 in \cite{2016_Stechlinski_Barton_CDC}, with directions matrix $\M=\II_{n_s n_u}$ since in that case 
$$\pmb{\Psi}(t)=[\psi_{(1)}(t) \quad \cdots \quad \psi_{(n_s)}] \otimes \II_{n_u}=\coord_{(i)}^{\rm T} \otimes \II_{n_u} \quad \forall t\in(\tau_{i-1},\tau_i].$$ 
\end{proof}
 
\begin{remark}
The assumptions in Theorem \ref{thm:cdc}, including existence and regularity of a solution of \eqref{NonSmooth_DAE1}-\eqref{NonSmooth_DAE2} on $[t_0,t_f]$, \ps{are} in line with direct optimal control methods for smooth systems, except of course for the allowance of $\phi,\h,\g$ to be nonsmooth. In fact, the function $\h$ is permitted to be discontinuous with respect to $t$ in the theory in \cite{2016_Stechlinski_Barton_CDC}, and both $\h$ and $\g$ may depend on extra problem parameters, although we omit these generalizations (and use $\M=\II_{n_s n_u}$) to make the exposition clearer.
Note also that if $\phi, \h, \g$ are $C^1$,  $\Tilde{\textbf{X}}(t)=\frac{\partial \Tilde{\x}}{\partial \pp}(t;\pp_0), \Tilde{\textbf{Y}}(t)=\frac{\partial \Tilde{\y}}{\partial \pp}(t;\pp_0)$, $\mumu=\nabla \phi(\pp_0)$ in Theorem \ref{thm:cdc} (hence classical results are recovered). Lastly, we note that the objective functional ${\phi}$ in \eqref{NonSmooth_DAEobj} above is in Mayer form, but  Bolza/Lagrange forms of the objective functional are also permitted with the theory in Theorem \ref{thm:cdc}  since Bolza/Lagrange can be transformed into Mayer form (see, e.g.,  \cite{kirk2004optimal}), as we demonstrate in the next example.
\end{remark}

\normalsize

\begin{figure}[ht!]
\centering
\begin{tikzpicture}[node distance=1.5cm]

\node (in1) [io] {Input: Initial states and control input};
\node (pro2) [process, right of=in1,xshift=2.5cm] {Solve DAE system and sensitivity system  using $\pp$};
\node (dec1) [decision, right of=pro2, xshift=3.5cm] {$\phi(\pp)$ minimized?};

\node (pro2a) [process, below of=dec1, yshift=-2.5cm] {Solve DAE using $\pp$ to calculate states};

\node (pro2b) [process, above of=dec1, yshift=2.2cm] {Update $\pp$ using a generalized gradient element of $\phi(\pp)$};
\node (out1) [io, left of=pro2a,xshift=-3.5cm] {Return $\pp$ 
 and the states};

\draw [arrow] (in1) -- (pro2);
\draw [arrow] (pro2) -- (dec1);
\draw [arrow] (dec1) -- node[anchor=south] {yes} (pro2a);
\draw [arrow] (dec1) -- node[anchor=north] {no} (pro2b);
\draw [arrow] (pro2b) -| (pro2);
\draw [arrow] (pro2a) -- (out1);

\end{tikzpicture}
\caption{\ha{Simple flow chart clarifying the sequential optimal control process.}}\label{Fig:flowchart}
\end{figure}
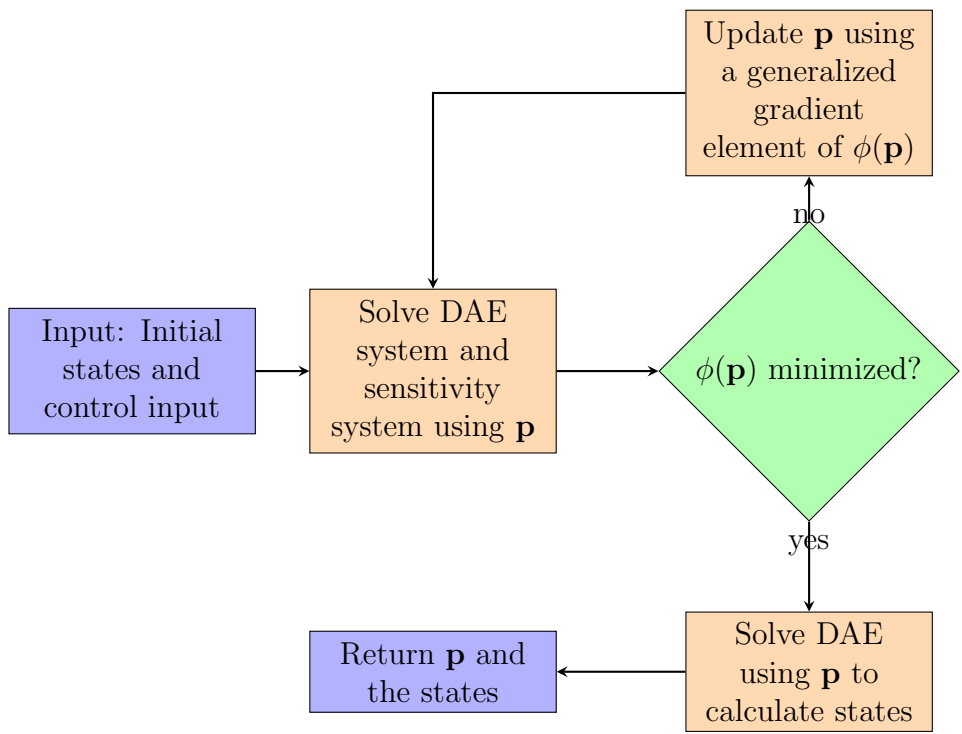
\ha{For an intuitive illustration of the sequential optimal control process, the reader can refer to the simple flow chart in Figure \ref{Fig:flowchart}.}
\pgs{\ha{We also} give an illustration of how the sequential LD-derivatives based optimal control method can be applied through the following example from \cite{Kelly2017} (in this case, using the ODE theory in \cite{khan2014generalizedoptimal} for comparative purposes).}

\begin{example}\label{ex:SiamBlock} 
 \pgs{ Consider a block that moves between the two points $x_1=0 \text{ and }x_1=1$, where $x_1$ is the position of the block,  within the time interval $[0,1]$, with the block's velocity, $x_2$, being  zero at both the initial time, $t_0=0$, and the final time, $t_f=1$. The control input $u$ to the system is the force applied to the left side of the block and no friction is assumed. We consider a nonsmooth objective functional as used  in \cite{Kelly2017} that minimizes the absolute work done. 
The optimal control problem can be then described as:}
\pgs{\begin{subequations}\label{SIAM_Example_Eq:1}\begin{align}
    \min\limits_u \; &\phi(u)=\int_{0}^{1}{|u(\tau)x_2(\tau)|d\tau}\label{siamobj_1},\\
    \text{s.t. }&\dot{x}_1 = x_2=:h_1,\qquad  x_1(0) = 0, \qquad x_1(1) = 1,\label{siamdynamics_1}\\
    &\dot{x}_2 = u=:h_2, \qquad  x_2(0) = 0, \qquad x_2(1) = 0\label{siamdynamics_2},
\end{align}
\end{subequations}}
\pgs{In using the sequential method, the continuous-time problem is transformed into an NLP by discretizing the control according to Eq.\ \eqref{control_param_full} and having the system dynamics in Eqs.\ \eqref{siamdynamics_1}-\eqref{siamdynamics_2}  converted to a set of constraints applied to the states and control at  discrete time points. A numerical ODE solver is used to compute the values of the states at these time points and, in the final step, the NLP solver is provided with an initial guess in the form of  reference control and states.}

\pgs{In order to apply the framework presented in \cite{khan2014generalizedoptimal}, we need the objective functional \eqref{siamobj_1} to be in the Mayer form. This can be done by introducing a new state variable $x_3(t)$, and equating the time derivative of that new variable to the integrand of the objective functional, i.e.
\vspace{-0.5em}
\begin{align}\label{eq:SiamWorkObjective}
    \dot{x}_3(t)=|u(t)x_2(t)|=:h_3,
\end{align}
with $x_3(0)=0$. So $n_x=3$ and $n_u=1$ and now the objective functional becomes 
\begin{align}\label{eq:SiamWorkObjectivefunctional}
  \min\limits_u\text{ } \phi(u)=\min\limits_u\text{ }x_3(t_f),
  \end{align}
which is in Mayer form. Parameterizing $u$ as a piecewise constant according to \eqref{control_param_full},  i.e., $u(t)=u_{(i)}$ on $(\tau_{i-1},\tau_i]$, $i=1,\ldots,n_s$, with $n_s=100$,  
the LD-derivative sensitivity function $\Tilde{\textbf{X}}(t)=(\Tilde{\textbf{X}}_1(t), \Tilde{\textbf{X}}_2(t), \Tilde{\textbf{X}}_3(t)) \in \real^{n_x \times n_s n_u}=\real^{3 \times 100}$ solves
\begin{align*}
    \dot{\textbf{X}}_1(t)&=\textbf{X}_2(t), \quad \forall t\in[0,1],\\
    \dot{\textbf{X}}_2(t)&=\coord_{(i)}^{\rm T}, \quad  \forall t\in(\tau_{i-1},\tau_{i}],\\   
    \dot{\textbf{X}}_3(t)
    &=\mathrm{fsign}\bigl({u}_{(i)} \Tilde{\textbf{X}}_2(t),\Tilde{x}_2(t) \coord_{(i)}^{\rm T}+{u}_{(i)} {\textbf{X}}_2(t)\bigl)\times \bigl( \Tilde{x}_2(t) \coord_{(i)}^{\rm T}+{u}_{(i)} {\textbf{X}}_2(t) \bigl),\;  \forall t\in(\tau_{i-1},\tau_{i}],\\   
    {\textbf{X}}_1(0)&=\zero_{1 \times n_s}, \quad {\textbf{X}}_2(0)=\zero_{1 \times n_s}, \quad {\textbf{X}}_3(0)=\zero_{1 \times n_s},
\end{align*} 
where  the LD-derivative rule
$$[\text{abs} \circ f]'(z;\M)=\text{abs}'(f(z);f'(z;\M))=\text{fsign}(f(z),f'(z;\M)) \times f'(z;\M)$$
has been used  with $z=(u,x_2)$ and $f(z)=u x_2$ here. Then, according to Theorem \ref{thm:cdc} (with algebraic equations absent, i.e., using \cite[Theorem 6 and Corollary 7]{khan2014generalizedoptimal}) for some reference control parameters $\pp_0=(u_{(1)},\ldots,u_{(n_s)})$,
\begin{align}
\begin{split}
    \etaeta=\Tilde{{\textbf{X}}}_3(t_f) \in \partial \phi(\pp_0).
\end{split}
\end{align}
} 
In Figure \ref{SIAM_BlockExample_all}, the results of the LD-derivatives based sequential method (above) is compared to/validated against the slack variables approach used in \cite{Kelly2017} for this same problem, and a smoothing approach using the following approximation of the absolute-value function \textcolor{black}{(as provided in \cite{Kelly2017})}: 
\begin{align} \label{smooth approximation_eq}
    |x| \approx x\tanh\left(\frac{x}{\alpha}\right),
\end{align}
for small $\alpha>0$.

\begin{figure}[h!]
    \centering
    {{\includegraphics[width=1\columnwidth]{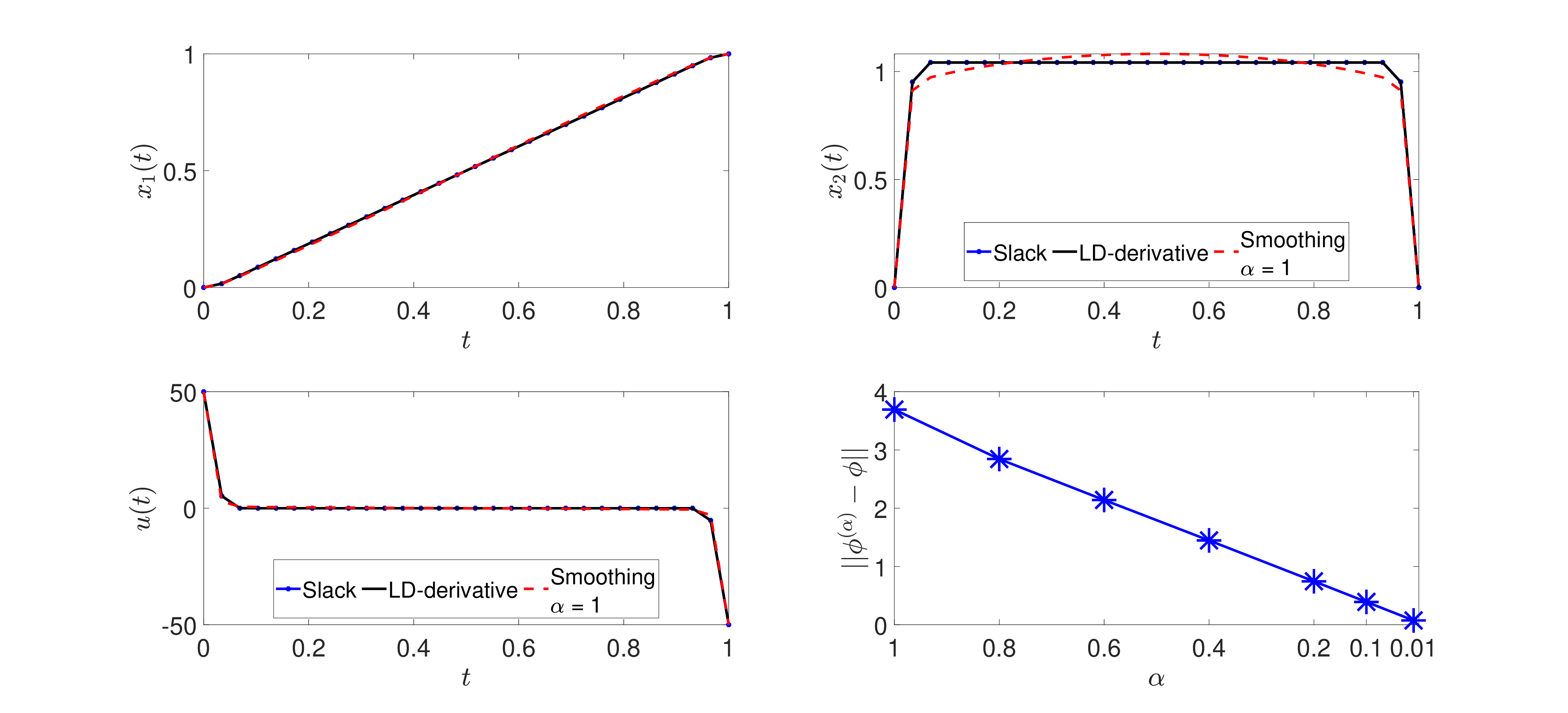}}}%
    \caption{\ha{This figure shows the results of Example \ref{ex:SiamBlock} using the objective functional in \eqref{eq:SiamWorkObjectivefunctional} (with piecewise constant control on $n_s=100$ subintervals of $[0,1]$)  via slack variables, smoothing approximation    
     and  LD-derivatives based sequential method. In the left panel, we see the optimal trajectory for the position $x_1(t)$ (upper panel) and the velocity $x_2(t)$ (lower panel). In the right panel, we show the control input signal $u(t)$ (upper panel) similar to the states in the left panel. In addition, we compare the 2-norm difference between the smoothed objective functional, $\phi^{(\alpha)}$, evaluated at different values of the smoothing parameter $\alpha$, and the nonsmooth objective functional $\phi$.}} %
     \label{SIAM_BlockExample_all}%
\end{figure}


It can be seen from the results \ha{in the upper right panel} of Figure \ref{SIAM_BlockExample_all} that the LD-derivatives based method produces an accurate solution when compared with the solutions from the slack variables and the smoothing approximation method (with a proper choice of the smoothing parameter $\alpha$), while avoiding the limitations of both of these methods (computational complexity of the slack variables and arbitrariness of smoothing approximation). We also used different values of the smoothing parameter $\alpha$ to show the dependency of this method's accuracy on the smoothing parameter, e.g., choosing $\alpha = 1$ provides a much more accurate solution than choosing $\alpha = 5$. \ha{In the bottom of the right panel of Figure \ref{SIAM_BlockExample_all}, we see the 2-norm of difference between the objective functional, $\phi^{(\alpha)}$, evaluated at different values of the smoothing parameter $\alpha$, and the nonsmooth objective functional $\phi$. As can be seen, a decrease in the value of $\alpha$ results in a better accuracy for the smoothing approximation. 
However, it remains unclear how to choose $\alpha$ to achieve a desired error, and it remains unclear in general whether such a smoothing approximation converges to the ``true'' solution. This, in part, motivates us to treat the nonsmooth problem directly (i.e., an LD-derivatives based approach).}
\end{example}

\section{\ha{Unifying and Generalizing WTPS Modeling and Control in Regions 1-3 using Nonsmooth Functions
}}\label{section:III} 
 
\pgs{Nonsmoothness in the WTPS  arises in three ways: (i) nonsmoothness in the parameters can be present due to extreme variations or situations, which can be modeled as nonsmooth profiles \cite{eisa2021sensitivity}; (ii) nonsmoothness in the control input;  and (iii) nonsmoothness from unifying and generalizing the WTPS modeling framework. In this part, we provide a unified and generalized WTPS control problem and modeling framework that is operable for all Regions 1-3 (as illustrated in Figure \ref{PmechWref}), by introducing nonsmooth RHS functions and objective functional to the model. Afterward, \textcolor{black}{in Section \ref{section:IV},}  we treat the introduced nonsmoothness directly using the  LD-derivatives based sequential optimal control method  outlined in Section \ref{section:II}.}   
\subsection{\ha{WTPS DAE Framework}}
First we provide the WTPS DAE model \textcolor{black}{ we use in this study (from \cite{eisa2019modeling,eisa2019nonlinear}) for the reasons (i)-(iv) we stated in subsection \ref{subsection:IC}. The reader is directed to the book chapter \cite{eisa2019nonlinear} for very detailed steps on the model derivations based on accurate physical/industrial considerations. Additionally, in said book chapter \cite{eisa2019nonlinear} it is mentioned how different control conditions and design considerations can amend the model for the user's need.} 
For example, one can take different settings for the reactive power control design, among many other options. However, with the generalization and unification in mind, and the steps provided in this section, one should be able to build on the results here to construct other particular frameworks customized to their problem of interest. For \textcolor{black}{our}  control problem, we consider the control input $\uu$ to be the pitch angle $\theta$, i.e.\ $u = \theta$. \textcolor{black}{Additionally}, we consider the two mass \textcolor{black}{model representation of electric and mechanical torques, \eqref{model1}-\eqref{model3}, as this is known, physically, to provide more accurate representation \cite{eisa2019modeling,eisa2019nonlinear} of the conversion of mechanical torque (due to turbine blades rotation via wind) into electric torque (provided to the generator). Moreover, we consider the power control \eqref{model4}-\eqref{model5} to operate without optional/extra \cite{eisa2019modeling,eisa2019nonlinear} inertia controls, frequency-bus control, or power storage. Also, the reactive power control is represented by the power factor mode \eqref{model6}-\eqref{model7} \cite{eisa2019modeling,eisa2019nonlinear}. Furthermore, the electric variables representing the active and reactive voltage/current delivered to the grid are represented by \eqref{model8}-\eqref{model10}\cite{eisa2019modeling,eisa2019nonlinear}. Lastly, the terminal voltage (at the connection between the wind turbine and the grid) is represented by the infinite-bus model \cite{eisa2019modeling,eisa2019nonlinear}, commonly used in literature for general testing of power sources interaction with the grid; this can be written algebraically based on Kirchhoff's laws as in \eqref{model11}}. Hence, the WTPS DAE system is given as follows:
\begin{subequations}\label{WTPS_model}
	\begin{align}
	&\dot{w}_g= \frac{1}{2H_g}\left[- \frac{P_{elec}}{w_g+w_0}-D_{tg}(w_g-w_t)-K_{tg} \Delta\theta_m\right] =: h_1,\label{model1}\\
	&\dot{w}_t=\frac{1}{2H}\left[\frac{P_{mech}}{w_t+w_0}+D_{tg}(w_g-w_t)
	+K_{tg} \Delta\theta_m\right] =: h_2,\label{model2}\\
	&\Delta\dot{\theta}_m = w_{base}(w_g-w_t)=: h_3,\label{model3}\\
	&\dot{f}_1 =w_g+w_0-w_{ref}=: h_4,\label{model4}\\
	&\dot{P}_{inp} =  \frac{1}{T_{pc}}\big[(w_g+w_0)(K_{ptrq}(w_g+w_0-w_{ref}) 
	+ K_{itrq}f_1)-P_{inp}\big]=: h_5,\label{model5}\\ 
	&\dot{P}_{1elec} = [P_{elec}-P_{1elec}]/T_{pwr}=: h_6,\label{model6}\\
	&\dot{V}_{ref}= K_{Qi}[Q_{cmd}-Q_{gen}]=: h_7,\label{model7}\\
	&\dot{E}_{qcmd} = K_{vi}[V_{ref}-V ]=: h_{8},\label{model8}\\
	&\dot{E}_{q} =  [E_{qcmd}-E_q]/0.02=: h_{9},\label{model9}\\
	&\dot{I}_{plv} = \left[ P_{inp}/V-I_{plv}\right]/0.02=: h_{10},\label{model10}\\
	&0=\left(V^2\right)^2-\left[2(P_{elec}R+Q_{gen}X)+E^2\right]V^2
	+\left(R^2+X^2\right)\left(P_{elec}^2+Q_{gen}^2\right)=: g,\label{model11}
	\end{align}
\end{subequations}
\normalsize
where $Q_{gen}= V(E_q-V)/X_{eq}$, $Q_{cmd}=\tan(PFE)P_{1elec}$, PFE is small angle, $w_{ref}$ is given by \eqref{w_ref},  $P_{elec}=I_{plv}V$, 
\begin{align}
\begin{split} 
    P_{mech}&=\frac{1}{2}C_p(\lambda,\theta)\rho A_r v_{wind}^3
    =\frac{1}{2} (\sum_{i=0}^{4} \sum_{j=0}^{4} \alpha_{i,j} \theta^i\lambda^j) \rho A_r v_{wind}^3,
    \end{split}
\end{align}
\normalsize
where $v_{wind}$ is the wind speed 
and $\lambda = \frac{(w_t+w_0)}{v_{wind}}$ is the tip ratio. 
The system   in \eqref{WTPS_model} has differential state variables 
$\x=(w_g,w_t,\Delta\theta_m,f_1,P_{inp},P_{1elec},V_{ref},E_{qcmd},E_q,I_{plv})$
(i.e., $n_x=10$)
and   algebraic state variable $y=V$
(i.e., $n_y=1$) representing the terminal voltage, where $w_g,w_t,\Delta\theta_m,f_1,P_{inp},P_{1elec}$, $V_{ref},E_{qcmd},E_q$, and $I_{plv}$ are the generator speed, the dynamic turbine speed, the integral of difference between $w_t$ and $w_g$, the integral of differences of speeds, the power order, the filtered electrical power, the reference voltage, the reactive voltage command, and the generator reactive variable, respectively.  
Denotation of all of the state variables and the parameters, as well as the values of the parameters are discussed and taken from \cite{eisa2019modeling,eisa2019nonlinear}, see  \ref{appendix}. 
\subsection{\ha{Unifying Regions: Nonsmooth Wind Profiles and Shaft Speed}}
We now introduce the wind profiles, representing wind speed variations, that we use to test/analyze the optimal control response in studying WTPSs. For that, we first consider a nonsmooth ramp wind profile which represents changes in the wind speed level (low slope for slow transitions and high slope, near discontinuity, for extreme and fast changes). This nonsmooth generalized ramp wind speed profile was also introduced in \cite{eisa2021sensitivity} as:
\begin{align}\label{v_wind,ramp}
    v_{wind,1}(t) = v_0+m(\max(0,t-t_{on})-\max(0,t-t_{off})),
\end{align}
\normalsize
where $v_0,m,t_{on},t_{off}$ are the initial wind speed before the start of the ramp, the slope of the ramp, the time at which the ramp starts, and the time at which the ramp ends, respectively. For wind gusts or disturbances, Gaussian wind speed profiles can be used (with large variance for calm gusts and very small variance, such that the wind profile will be close to a delta function, for extreme gusts), and can be expressed as:
\begin{align}\label{v_wind,Gaussian}
    v_{wind,2}(t)=v_0+\frac{1}{\sigma \sqrt{2\pi}}e^{\frac{-1}{2}(\frac{t-\mu}{\sigma})^2},
\end{align}
\normalsize
where $v_0$ is the initial wind speed, $\mu$ and $\sigma$ are the mean value and the standard deviation of the normal distribution, respectively.

Next, we note that having the system \eqref{NonSmooth_DAE1}-\eqref{NonSmooth_DAE2} applicable to all Regions of operations \ha{(with Regions 1-2 representing low and moderate wind speeds that are below the rated wind speed, and Region 3 representing high wind speeds above the rated wind speed \cite{njiri2016state}, as illustrated in Figure \ref{PmechWref})} will require some unification and generalization of conditions relevant to Regions 1-3. This unification and generalization necessarily introduces nonsmoothness in the dynamics and the objective functional of the control problem.
 \begin{figure}[!ht]\centering
\centerline{\includegraphics[width=0.8\columnwidth]{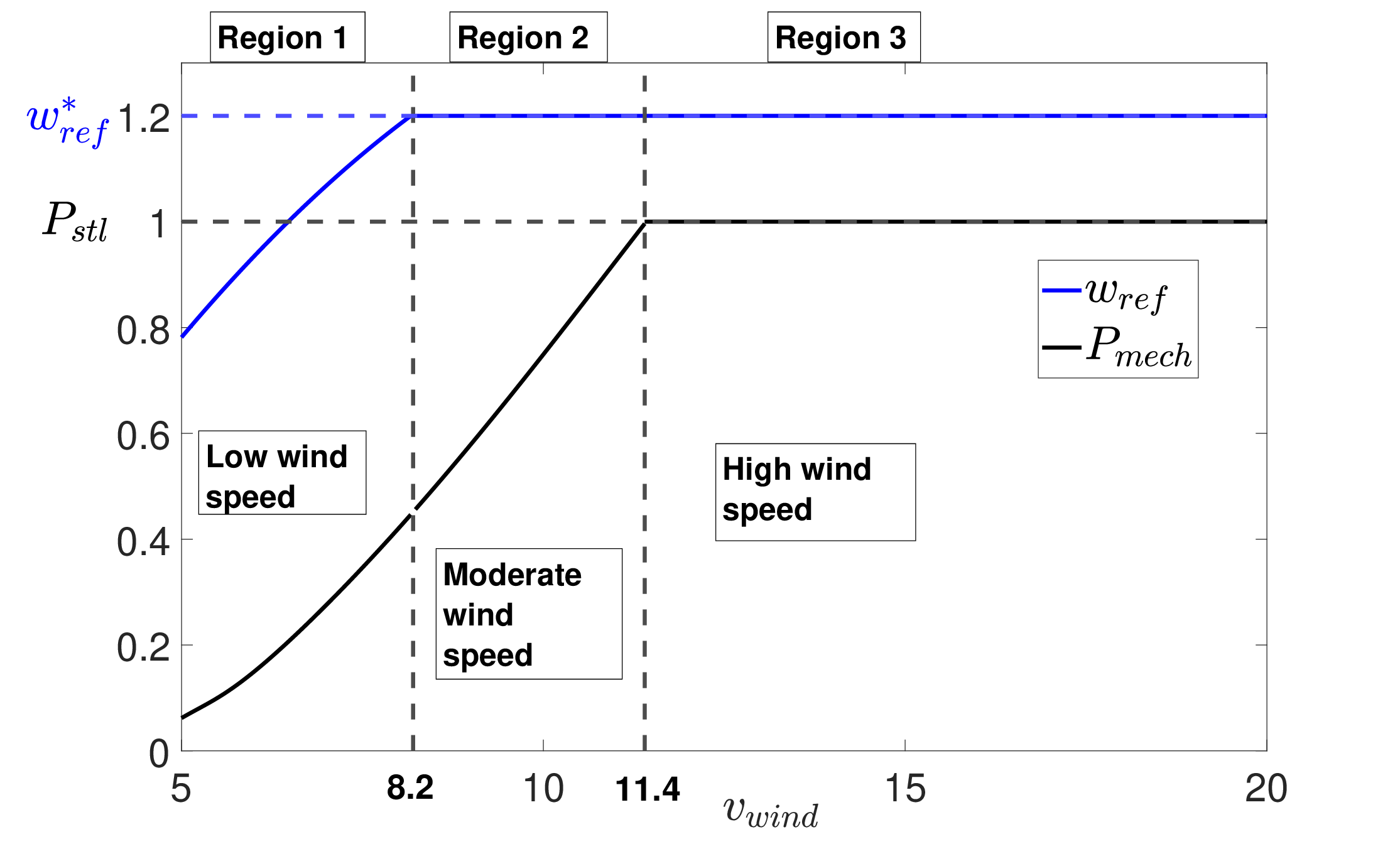}}
\setlength{\belowcaptionskip}{0.1pt}
\caption{How $w_{ref}$ and $P_{mech}$ change as $v_{wind}$ changes. There are three Regions of operation corresponding to different wind speed ranges. $w_{ref}$ transitions through nonsmoothness between Regions 1 and 2 at $v_{wind}=8.2$ \text{ m/s} \cite{eisa2019modeling,eisa2019nonlinear} and $P_{mech}$ transitions through nonsmoothness between Regions 2 and 3 at about $v_{wind}=11.4$ \text{ m/s} \cite{eisa2019modeling,eisa2019nonlinear}, respectively.}\label{PmechWref}
\end{figure}
We start with the reference shaft speed, $w_{ref}$, which controls the rotor. It can be modeled as a quadratic polynomial function of the generated electric power, $w_{ref}=-0.75P^2_{elec}+1.59P_{elec}+0.63$, in Region 1 of operations (low wind speeds) due to physical considerations \cite{eisa2019nonlinear,eisa2019modeling}. However, once the wind speeds are high enough (Regions 2-3), $w_{ref}$ should be kept at its rated value (constant) $w_{ref}=w^*_{ref}$, e.g. $w^*_{ref}=1.2$ for $v_{wind}\ge 8.2 \text{ m/s}$ \cite{eisa2019nonlinear,eisa2019modeling}. Thus, unifying $w_{ref}$ in Regions 1-3 can be done by introducing the nonsmooth min function:
\begin{align} \label{w_ref}
    w_{ref} = \min(-0.75P^2_{elec}+1.59P_{elec}+0.63,w^*_{ref}).
\end{align}
We can see from Eq.\ \eqref{w_ref} that for lower sind speeds ($v_{wind}<8.2 \text{ m/s}$), the quadratic expression in the left argument of the min function is less than $w^*_{ref}$, i.e., $-0.75P^2_{elec}+1.59P_{elec}+0.63<w^*_{ref}$, so the output of the min function will be $w_{ref}=-0.75P^2_{elec}+1.59P_{elec}+0.63$. At $v_{wind}=8.2 \text{ m/s}$, we have the quadratic expression exactly equal to $w^*_{ref}$, i.e., $-0.75P^2_{elec}+1.59P_{elec}+0.63=w^*_{ref}$, hence, we have $w_{ref}=-0.75P^2_{elec}+1.59P_{elec}+0.63=w^*_{ref}$. Finally, at $v_{wind}>8.2 \text{ m/s}$, the quadratic expression gives a value higher than $w^*_{ref}$, hence the output of the min function becomes $w_{ref}=w^*_{ref}$. It is common to take $w^*_{ref}=1.2$ \cite{eisa2019modeling,eisa2019nonlinear}. 

\subsection{\ha{Unifying The Control Objectives Over all Regions: A Nonsmooth Objective Functional}}
As for generalizing an expression for $P_{mech}$ for all wind speeds (Regions 1-3 in Figure \ref{PmechWref}), it is not as straightforward as what we did with $w_{ref}$; $P_{mech}$ is the objective itself that the WTPS should be designed to achieve \cite{eisa2019modeling,eisa2019nonlinear}: (i) maximization of $P_{mech}$ when physically it is not possible to produce the rated level, $P_{stl}$,  which is usually normalized to be 1 Per Unit \cite{eisa2019modeling,eisa2019nonlinear}, due to lower wind speeds (Regions 1-2); and (ii) maintaining $P_{mech}$ at the rated level, i.e., $P_{mech}=P_{stl}$ when physically $P_{mech}$ can exceed $P_{stl}$ (Region 3). Hence, unification of Regions 1-3 can be achieved by generalizing the control objective itself, i.e., the objective functional. In fact, the authors in \cite{johnson2012assessment} introduced an objective function that achieves the above-mentioned control objective. Within the context of an optimal control problem, their expression of the objective functional can be expressed as:
\begin{align}\label{maxminobjective}
   \max\limits_{u} \;\int_{t_0}^{t_f}{\min(P_{mech}  (\x(t),u(t)),P_{stl})}dt.
\end{align}
\normalsize
In \eqref{maxminobjective} whenever $P_{mech}$ is less than the rated power, $P_{stl}$, the min function will return the $P_{mech}$ argument which leads to maximization of power production; if $P_{mech}$ is more than the rated power $P_{stl}$, then the min function will return the constant $P_{stl}$ argument.
One major problem with \eqref{maxminobjective} is that it does not have an isolated maximum when the min function returns the constant $P_{stl}$ argument (one side of the function is constant). To resolve this issue, we introduce a well-defined new objective functional which achieves the same objective of \eqref{maxminobjective} but with an isolated maximum. This generalized objective functional that will unify the control objective over Regions 1-3 is expressed as:
\begin{align}\label{eq:objfunWPTSintegral}
    \begin{split}
        \max\limits_{u} \; \int_{t_0}^{t_f}&\omega(\x(t),u(t))dt=\phi(u),
    \end{split}
\end{align}
\normalsize
where
\begin{align}\label{eq:objfunWPTSintegrand}
    \begin{split}
    \omega(\x(t),u(t))
    &={\min({P_{stl}},P_{mech}(\x(t),u(t)))}\\
    &-{\min(0,{P_{stl}}-P_{mech}(\x(t),u(t)))}
        \times ({P_{stl}}-P_{mech}(\x(t),u(t)))=:h_{11}.
    \end{split}
\end{align}
\normalsize
In relevance to our results in the next subsection where we introduce our optimal control formulation, one needs to have the objective functional in the Mayer form. To convert \eqref{eq:objfunWPTSintegral} from Lagrange to Mayer form, we introduce an auxiliary variable $x_{aux}$ as:
\begin{align} \label{auxillary state}
    \begin{split}
        &\dot{x}_{aux}(t)=h_{11}, \qquad x_{aux}(t_0)=0.
    \end{split}
\end{align}
\normalsize
With this definition, we get $\phi(u)=\int_{t_0}^{t_f}\omega(\x(t),u(t))dt=x_{aux}(t_f)$, and hence \eqref{eq:objfunWPTSintegral} is replaced by  
\begin{align} \label{objectivefunctionalWTPS}
    \begin{split}
        \max\limits_{u} \text{ } \phi(u)=\max\limits_{u} \text{ } x_{aux}(t_f).
    \end{split}
\end{align}
This objective functional, again, is generalized over Regions 1-3 and unifies both the control objectives (maximizing power production at lower wind speeds and maintaining the power production at the rated level for higher wind speeds). To elaborate further, observe that in \eqref{eq:objfunWPTSintegrand} whenever $P_{mech}$ is less than the rated power \ha{(which happens at low wind speeds, where we need to maximize power production)}, then what we get from Eq.\ \eqref{eq:objfunWPTSintegrand} is $P_{mech}-0$,
which leads to the maximization (minimization with negative sign) of power production. However, if $P_{mech}$ is greater than the rated power $P_{stl}$, then Eq.\ \eqref{eq:objfunWPTSintegrand} returns $P_{stl}-(P_{stl}-P_{mech})^2$,
where the function is not constant anymore (hence, we resolve the issue of \eqref{eq:objfunWPTSintegral}) and then the difference $(P_{stl}-P_{mech})^2$ is minimized, keeping $P_{mech}$ at the rated level $P_{stl}$. \ha{This generalization of the objective functional over all wind speeds is achieved thanks to the presence of nonsmooth functions in the objective functional.}

\section{LD-Derivatives Based Method: A Novel Optimal Control Approach for WTPSs}\label{section:IV}

Note that the ODE associated with $h_2$ contains $P_{mech}$, and thus the piecewise control $u=\theta$ (i.e., $n_u=1$). We also note that for $w_{ref}$ and the objective functional, we use the nonsmooth expressions in \eqref{w_ref} and \eqref{objectivefunctionalWTPS}, respectively, that were  introduced in the previous section. Recall from \eqref{auxillary state} that to have the objective functional in Mayer form we need to introduce an extra differential state variable, $x_{aux}$, to the system of Eqs.\ in \eqref{WTPS_model}. Hence, we get a total of eleven differential state variables and one algebraic state variable (i.e., $n_x=11$ and $n_y=1$) and the time derivative of this extra state variable will appear coupled with the system \eqref{WTPS_model}  
as given in Eq.\ \eqref{auxillary state}. 
With the control parameters \ha{$\pp=(\bm{u}_{(1)},\dots,\bm{u}_{(n_s)})$} in the discretization  
\begin{equation}\label{eq.control_param_full.latersection} 
\ps{
\uu(t) 
  =\textbf{u}_{(i)} \quad \text{ if } t \in (\tau_{i-1},\tau_{i}],
}
\end{equation} 
as the decision variables, we used the MATLAB\textsuperscript{\textregistered} built-in NLP solver, \texttt{fmincon}, to solve the NLP problem arising from this problem. (We note that $\tau_i$ can also be taken as decision variables for more accuracy but in our case we only consider controlling $u_{(i)}$.)
Because participating functions are nonsmooth, conventional methods (such as \texttt{fmincon}) may fail.
To overcome this, the LD-derivative sensitivity functions as well as the generalized gradient  elements of the objective functional obtained via Theorem \ref{thm:cdc}  are supplied to the solver \texttt{fmincon} (which does not require such inputs in the smooth setting, as this method can generate these objects itself given a smooth ODE/DAE control problem).

\pgs{Assuming that \eqref{WTPS_model} admits a  regular solution on $[t_0,t_f]$, i.e.\ \eqref{eq.clarkeregularity} is satisfied along the solution, it must hold that $\frac{\partial g}{\partial V} \neq 0$ along the solution since in this case $g$ is smooth and we have that $\pi_y \partial g=\{\frac{\partial g}{\partial V}\}$ (see Eq.\ (8) in \cite{eisa2021sensitivity}). In this case, Theorem \ref{thm:cdc} yields a generalized gradient of  \eqref{eq.OCP.gengradient} by building the auxiliary system \label{NonSmooth_Optimal_eq:4} whose solutions are LD-derivative sensitivity functions. This can be accomplished by applying LD-derivative rules, as outlined in Section \ref{subsection:IIA}.  
The sensitivity equations associated with $g$ and $h_j$, $j \notin \{4,5,11\}$, are the classical sensitivity equations:
\begin{align}
\dot{\textbf{X}}_{x_j}(t)
&=
\frac{\partial h_j}{\partial u} \coord_{(i)}^{\rm T}+
\frac{\partial h_j}{\partial \x} \textbf{X}(t)+
\frac{\partial h_j}{\partial V} \textbf{V}(t), \qquad t\in(\tau_{i-1},\tau_i]
\label{eq.sensitivities.smoothX},
\end{align} 
\normalsize
with partial derivatives  evaluated at $(t,u_{(i)},\Tilde{\x}(t),\Tilde{V}(t))$, e.g., $$\dot{\textbf{X}}_{E_{q}}(t)=\frac{1}{0.02}\left[\textbf{X}_{E_{qcmd}}(t)-
\textbf{X}_{E_{q}}(t)\right].$$ 
Similarly, the algebraic equation Eq.\ \eqref{model11} gives the following sensitivity equation:
\begin{align}
\zero_{1 \times n_s}
&=\frac{\partial g}{\partial E_{q}} \textbf{X}_{E_{q}}(t)+\frac{\partial g}{\partial I_{plv}} \textbf{X}_{I_{plv}}(t)+\frac{\partial g}{\partial V} \textbf{Y}(t).\label{eq.sensitivities.smoothY} 
\end{align}
\normalsize
For the nonsmooth functions $h_4,h_5,$ and $h_{11}:=\dot{x}_{aux}$,  we apply the theory of LD-derivatives to get the corresponding sensitivity equations, i.e., for $h_4$ we get:  
\begin{align}
\label{eq.sensitivities.h4} 
\begin{split}
&\dot{\textbf{X}}_{f_{1}}(t)=h_{4}'(t,u(t),\Tilde{\x}(t),\Tilde{V}(t);(\zero_{1 \times n_s},\coord_{(i)}^{\rm T},\textbf{X}(t),\textbf{Y}(t)))\\
&=\textbf{X}_{w_g}(t)-\left[w_{ref}\right]'\\
&=\textbf{X}_{w_g}(t)-\left[\min(-0.75P^2_{elec}+1.59P_{elec}+0.63,w^*_{ref})\right]'\\
&=\textbf{X}_{w_g}(t)-\lshiftmin ([q(P_{elec}) \quad (-1.5P_{elec}+1.59) P_{elec}'],[ w^*_{ref} \quad \bm{0}_{1\times n_s}]), \quad \forall t\in(\tau_{i-1},\tau_i],
 \end{split}
 \end{align}  
\normalsize 
where $q(P_{elec})=-0.75P^2_{elec}+1.59P_{elec}+0.63$ and $P_{elec}'=\frac{\partial P_{elec}}{\partial u}\coord_{(i)}^{\rm T} +\frac{\partial P_{elec}}{\partial \x}\textbf{X}(t)+\frac{\partial P_{elec}}{\partial V}\textbf{Y}(t)$. Similarly, for $h_5$ we have:
\begin{equation}\label{eq.sensitivities.h5}  
\begin{aligned}
\dot{\textbf{X}}_{P_{inp}}(t)
&=\frac{1}{T_{pc}} \left[ K_{itrq}(w_g+w_0)\right] \textbf{X}_{f_1}(t)
-\frac{\textbf{X}_{P_{inp}}(t)}{T_{pc}}+\bigr[ K_{ptrq}(w_g+w_0-w_{ref})+K_{itrq} f_1\\
&+K_{ptrq}(w_g+w_0) \bigr] \frac{\textbf{X}_{w_g}(t)}{T_{pc}}
-\frac{K_{ptrq}(w_g+w_0)}{T_{pc}}\left[w_{ref}\right]'.
\end{aligned}
\end{equation}
\normalsize
For $h_{11}$ we get:
\begin{align}\label{eq.sensitivities.h11}  
&\dot{\textbf{X}}_{x_{aux}}(t)\\
&=\left[ {\min(P_{stl},P_{mech})}-{\min(0,{P_{stl}}-P_{mech})} ({P_{stl}}-P_{mech})\right]' \nonumber\\
&=\lshiftmin([P_{stl} \quad \bm{0}_{1\times n_s}],[P_{mech}  \quad \frac{\partial P_{mech}}{\partial u}\coord_{(i)}^{\rm T} +\frac{\partial P_{mech}}{\partial \x}\textbf{X}(t)]) \nonumber \\
&-\lshiftmin([0 \quad \bm{0}_{1\times n_s}],[P_{stl}-P_{mech}\quad -\frac{\partial P_{mech}}{\partial u}\coord_{(i)}^{\rm T} -\frac{\partial P_{mech}}{\partial \x}\textbf{X}(t)]) \times({P_{stl}}-P_{mech}) \nonumber\\
&+{\min(0,{P_{stl}}-P_{mech})}(\frac{\partial P_{mech}}{\partial u}\coord_{(i)}^{\rm T} +\frac{\partial P_{mech}}{\partial \x}\textbf{X}(t)).
\end{align} 
\normalsize
Recall that the initial conditions for $\textbf{X}(t_0)$ are zero, and it follows from Eq.\ \eqref{NonSmooth_Optimal_eq:4} that $\textbf{Y}(t_0)$ must satisfy
\begin{align*}
\zero_{1 \times n_s n_u}
&=g'(t_0,\x_0,\y_0;(\zero_{1 \times n_s n_u},\zero_{n_x \times n_s n_u},\capy(t_0)))
=\frac{\partial g}{\partial V} \textbf{Y}(t_0)
\end{align*}
with $\frac{\partial g}{\partial V}\neq 0$ by regularity, so that $\textbf{Y}(t_0)=\zero_{1 \times n_s n_u}$ must hold. Finally, the generalized gradient $\bm{\mu}$ of the objective functional in \eqref{objectivefunctionalWTPS} can be calculated as: $
    \bm{\mu}^T= \textbf{X}_{x_{aux}}(t_f) \in \partial \phi(\pp_0). $}

\subsection{Simulation Results of LD-Derivatives Based Method}\label{subsection:IIIC}
In order to test the optimal control formulation presented in the previous subsections, we expose the WTPS to variations in the wind speed, $v_{wind}$. We consider three simulation cases for $v_{wind}$: the ramp profile in Eq.\ \eqref{v_wind,ramp}, the Gaussian profile in Eq.\ \eqref{v_wind,Gaussian}, and a wind profile produced from actual wind speed data (real-world measurements) that was used to validate the DAE model \eqref{WTPS_model} as illustrated in \cite{eisa2019modeling,eisa2019nonlinear,eisa2018wind}. We note that we have carefully chosen our profiles such that it gurantees the WTPS will transition between Regions 1-3 in Figure \ref{PmechWref}. Hence, we test directly the effectiveness of our approach when the system passes through unavoidable nonsmoothness.  \ha{We note that the reader can refer to this GitHub repository \cite{githubwtps} for the MATLAB files used to produce the results discussed in this subsection.}

In the upper panel of Figure \ref{v_0=10, Ramp}, we show the first case of a ramp wind profile with an initial wind speed $v_0=10 \text{ m/s}$ and a final wind speed $v_f=12 \text{ m/s}$. We run the simulation from $t_0=18 \text{s}$ to $t_f=22 \text{s}$, with the number of time steps $n_s = 20$, and with the start and end times of the ramp at $t_{on}=19 \text{s}$ and $t_{off}=21 \text{s}$, respectively. The corresponding optimal control response, $u(t)$, to the change in $v_{wind}$ is depicted in the lower left panel, and the effect of the control response on $P_{mech}$ is depicted in the lower right panel. We notice that there is no change in $u(t)$ until about $v_{wind} = 11.4 \text{ m/s}$, which is the critical value of $v_{wind}$ at which there is nonsmoothness in the objective functional \eqref{objectivefunctionalWTPS} (i.e.\    $P_{mech}=P_{stl}$). After reaching that point, $u(t)$ increases to keep $P_{mech}$ at the rated power $P_{stl}=1$.

\begin{figure}[ht!]
\centering
\centerline{\includegraphics[width=0.8\columnwidth]{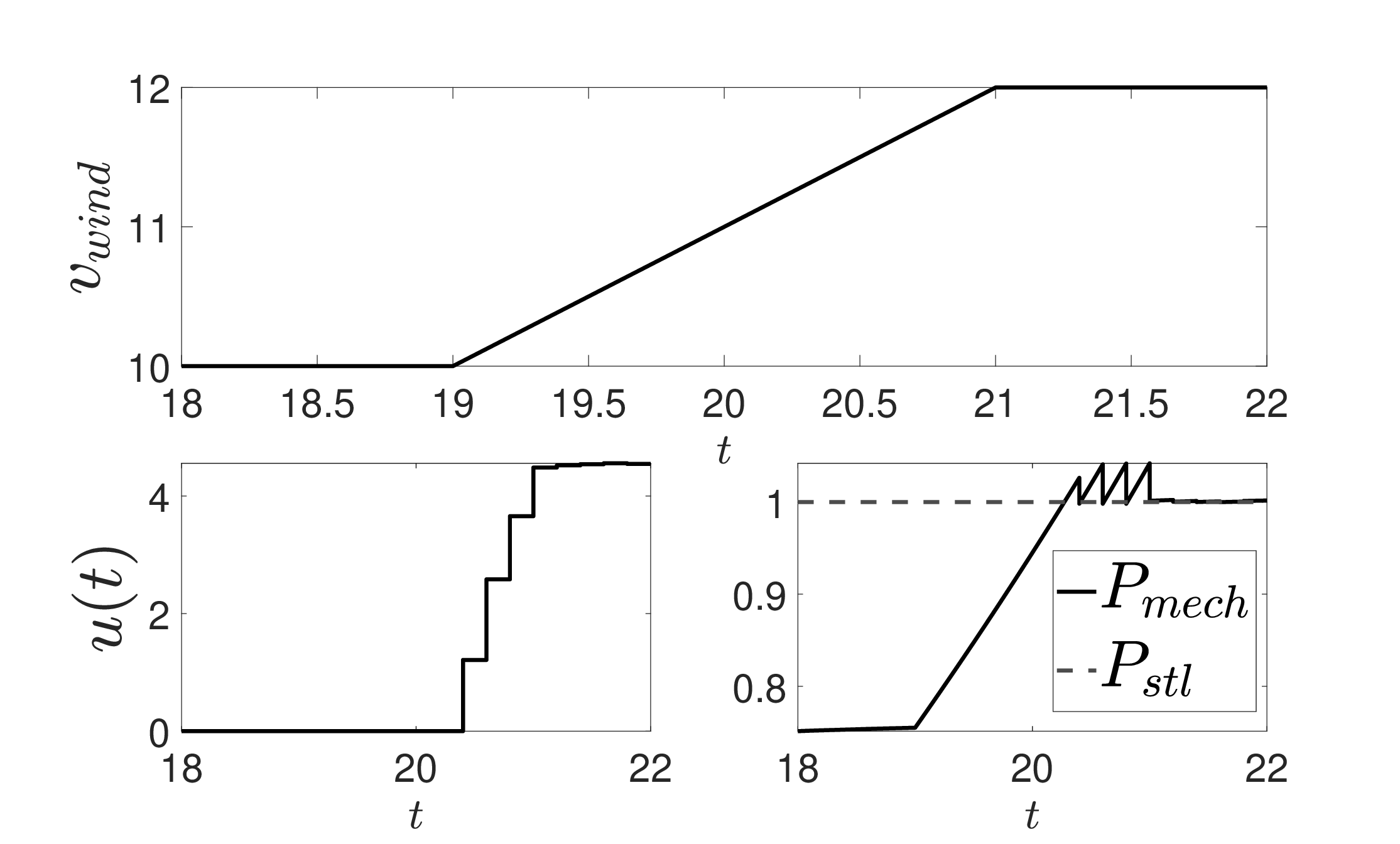}}
\caption{The upper panel shows the ramp profile of $v_{wind}$ with an initial wind speed $v_0=10 \text{ m/s}$, and a final wind speed $v_f=12 \text{ m/s}$. The lower left panel shows the optimal control response, $u(t)$, to the change in $v_{wind}$. The lower right panel shows the effect of the control response on $P_{mech}$.}\label{v_0=10, Ramp}
\end{figure}

In the upper panel of Figure \ref{v_0=10, Gaussian} we show the second case with a Gaussian wind profile. We run the simulation from $t_0=18 \text{s}$ to $t_f=22 \text{s}$, with the number of time steps ${n_s} = 20$, $v_0=10 \text{ m/s}$, mean $\mu=20$, and standard deviation $\sigma = 0.5$. The corresponding optimal control to the change in $v_{wind}$ is depicted in the lower left panel, and the  effect of the control response on $P_{mech}$ is depicted in the lower right panel. This change of $u(t)$ keeps the value of $P_{mech}$ close to the value of the rated power $P_{stl}=1$.

\begin{figure}[ht!]
\centering
\centerline{\includegraphics[width=0.8\columnwidth]{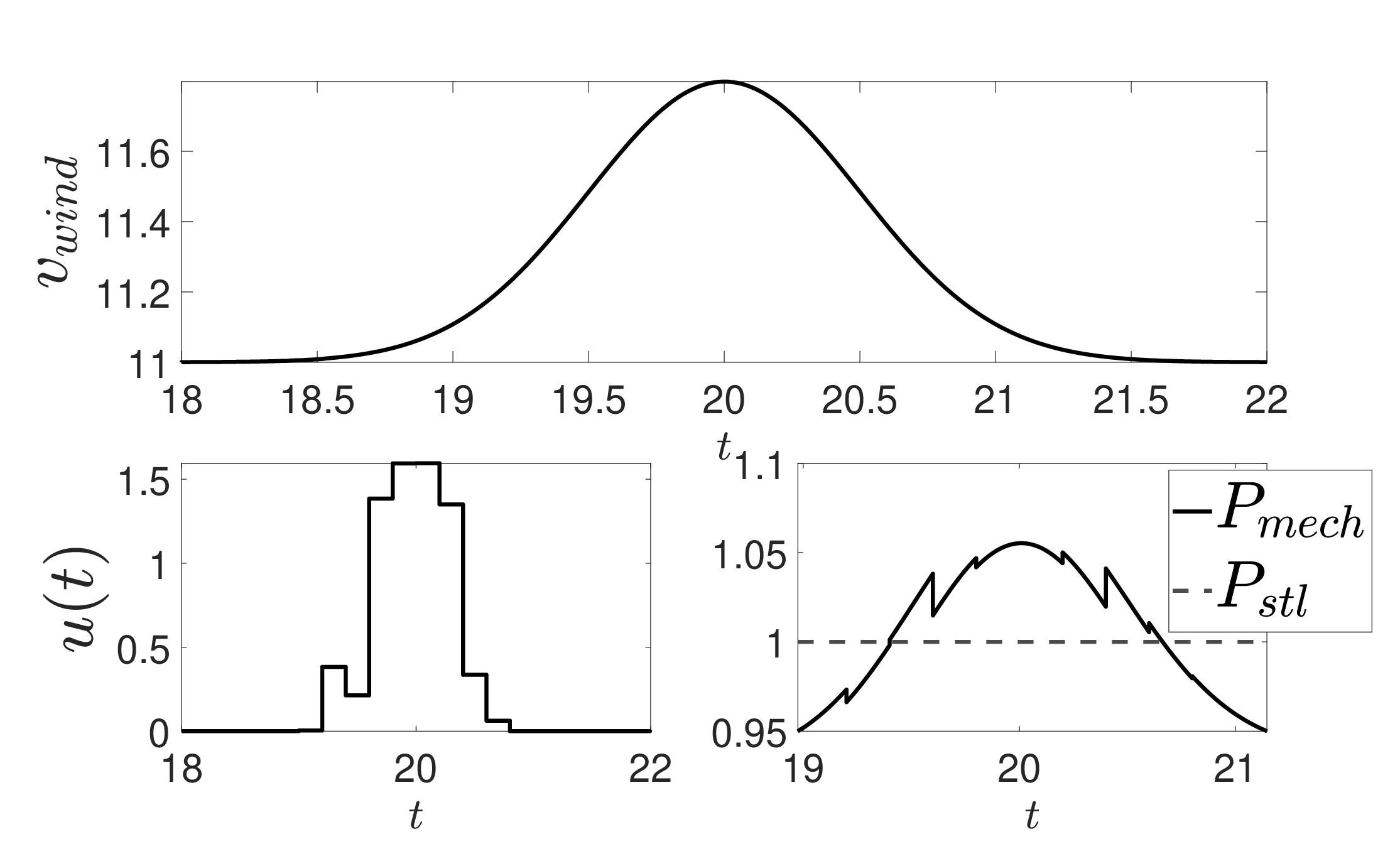}}
\caption{The upper panel shows the Gaussian profile of $v_{wind}$ with an initial wind speed $v_0=10 \text{ m/s}$, and standard deviation $\sigma = 0.5$. The lower left panel shows the optimal control response, $u(t)$, to the change in $v_{wind}$. The lower right panel shows the effect of the control response on $P_{mech}$.}
\label{v_0=10, Gaussian}
\vspace{0em}
\end{figure}

Finally, Figure \ref{RealWindData} shows actual real-world wind speed data used for data validation in \cite{eisa2019nonlinear,eisa2019modeling,eisa2018wind}. Here, we take the number of time steps to be $n_s = 425$. In Figure \ref{RealWindData}, the wind speed profile, $v_{wind}$, is shown as well as the corresponding control response and the resulting $P_{mech}$. We provide a close-up look of one Region of interest with very abrupt changes in wind speed. It is clear that the mechanical power is maximized for low $v_{wind}$ when $P_{mech}$ is less than the rated value $P_{stl}$. However, for higher wind speeds, the power is maintained at the rated value $P_{stl}=1$. This shows that the solution obtained using the sequential LD-derivative optimal control method maximizes the nonsmooth objective functional in a highly-turbulent wind profile where the system passes through points of nonsmoothness transitioning between Regions 1-3. Moreover, it can be seen that in this scenario the optimal control implementation was on a large time horizon of varying wind speeds, which demonstrates the effectiveness of the proposed framework and its general applicability to WTPSs. 

\begin{figure}[ht!]
\centering
\centerline{\includegraphics[width=0.8\columnwidth]{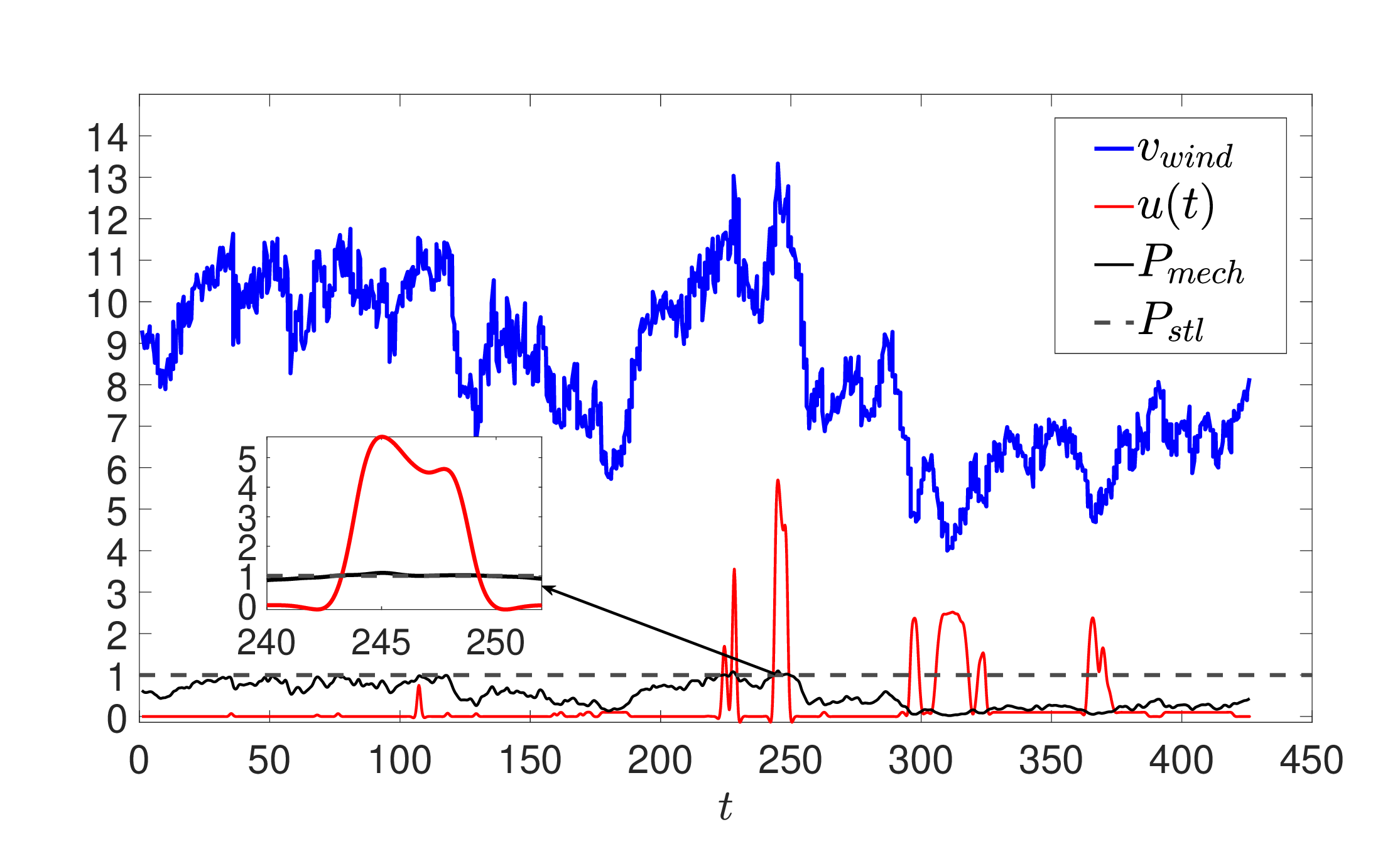}} 
\setlength{\belowcaptionskip}{0pt}
\caption{Our approach handling a real-world, high-turbulent wind data. The wind speed profile is shown as well as the corresponding control response and the resulting $P_{mech}$. The zoomed-in portion illustrates more clearly how the optimal control solution and the resulting $P_{mech}$.
}\label{RealWindData}
\end{figure}

Having verified the LD-derivatives based sequential nonsmooth optimal control method for the WTPS control problem via simulations, \ha{we provide the following observations, including the application potential of our optimal control method:}
\begin{enumerate}
    \itemsep-0.5em 
    \item It is noticeable that in general the optimal control trajectory seems to be close/matching in its shape to the wind profile which is useful information in control design.
    \item The current framework can adopt more control inputs (e.g., torque/tip-ratio based control inputs).
    \item Bounds and limiters as desired by the user, or similar to what is implemented in industry (e.g., \cite{eisa2019nonlinear}), can be easily added as extra constraints to this nonsmooth optimal control formulation.
    \item The current work can be extended  to involve reactive-power-control using the Q Droop function (see \cite{eisa2019modeling}) as a control input in fault cases.
    \item \textcolor{black}{For practical purposes, the proposed framework is useful in the design and analysis phase of control designs of WTPSs as it provides the optimal control trajectory (without linearization of the system or approximations of the nonsmooth functions present in the system and/or the objective functional). This can be then used to judge the effectiveness of other control methods (including real-time control methods). That is, the optimal control solution can serve as a reference of performance assessment to developing/proposed 
 control designs.}
    \item \textcolor{black}{The provided method can be implemented in near-real-time setting with the aid of high processing power and/or parallel processing handling the NLP.}  
\end{enumerate}
\ha{We also provide the following limitations:
\begin{enumerate}
 \itemsep-0.5em 
    \item We are not able to use second-order derivatives information within the optimal control problem due to the current absence of second order LD-derivatives theory in the literature.
    \item As is the case with other popular optimal control methods in the same vein, a suitable choice of initial conditions for the states and controls remains a challenge in our formulation.
\end{enumerate}}


In the next section, we make the case clearer for why the proposed nonsmooth LD-derivatives based sequential optimal control approach is needed as opposed to other alternatives.

\ha{\section{Comparison Between our Proposed LD-Derivatives Based Method and Alternative Approaches for the WTPS Control Problem}\label{section:V}}
In this section, \ps{we consider alternative approaches to solving the WTPS optimal control problem. In particular, instead of treating the nonsmoothness directly, like the proposed LD-derivative based method above,} we consider \ps{methods that avoid the nonsmoothness altogether}.
We broadly group possible alternatives into two approaches: (i) smoothing approaches (where \ps{nonsmooth functions are} approximated by \ps{smooth functions and standard methods are then  employed}) and (ii) ``naive'' approaches (where \ps{standard methods are immediately employed and} any \ps{theoretical/computational} issues stemming from the nonsmoothness are simply ignored). 

Before proceeding, we note that in our comparative results we consider the nonsmoothness present in the objective functional by limiting the wind speed range over Regions 2-3. We also interpolated the piecewise constant control solution for improved visualizations in all figures in this section. \ha{We note that for all simulations in this section, we have used the fmincon package in MATLAB\textsuperscript{\textregistered}, where we specify the following options:
\begin{itemize}
    \item We use the sequential quadratic programming (SQP) algorithm to solve the NLP problems in all of the approaches presented below.
    \item For the LD-derivatives based solutions, we provide LD-derivative information as the gradient information to be used in its corresponding NLP problem.
    \item We use the Broyden–Fletcher–Goldfarb–Shanno (BFGS) algorithm in all of the approaches presented below, which provides an approximation to the Hessian matrix in the NLP problems. In the case of the LD-derivatives based approach, the BFGS uses the first-order LD-derivative information. 
\end{itemize}}

\subsection{Comparison with a Smoothing Approach}\label{subsection:IVA}

We approximate the nonsmooth objective functional in \eqref{objectivefunctionalWTPS} (by replacing the nonsmooth expression $\omega$ in \eqref{eq:objfunWPTSintegrand})  using a known logarithmic approximation for the min function \cite{cook2011basic}. Thus, the smoothed version is:
\begin{align}\label{smoothingobjectivefucntionalWTPS}\begin{split}
\omega \approx \widehat{\omega}=\frac{\ln(e^{-NP_{stl}}+e^{-NP_{mech}})}{-N}-\frac{\ln(e^{0}+e^{-N(P_{stl}-P_{mech}))})}{-N}(P_{stl}-P_{mech}),
\end{split}
\end{align}
where $N$ is the smoothing parameter, with $\widehat{\omega} \to \omega$ as $N \to \infty$.  Now, we solve the WTPS optimal control problem in a similar manner to the previous section (our LD-derivatives based approach) \ps{but with a smooth optimal control solver after} using the smoothed objective functional above.
Two results are presented in Figures \ref{Smoothing, Ramp}-\ref{Smoothing, Gaussian}. Figure \ref{Smoothing, Ramp} shows the applied ramp wind profile and how different choices of $N$ affect the accuracy of the solution. The upper panel shows the profile of $v_{wind}$ with an initial wind speed $v_0=10 \text{ m/s}$, a final wind speed $v_f=12 \text{ m/s}$, start time and end time of the ramp at $t=19\text{s}$, $t=21\text{s}$, respectively. The lower left panel shows the optimal control response, $u(t)$, to the change in $v_{wind}$ using the introduced  LD-derivatives based sequential nonsmooth optimal control approach with Eq.\ \eqref{objectivefunctionalWTPS} as the objective functional, compared with the smoothed objective functional in Eq.\ \eqref{smoothingobjectivefucntionalWTPS} with different values of the smoothing parameters $N$. The lower right panel shows the effect of each of the control response trajectories on $P_{mech}$. The trajectory of $P_{mech}$ corresponding to the (optimal) trajectory of $u$ obtained using the LD-derivatives based approach effectively drives the dynamics as we expect, based on the nonsmooth objective functional in Eq.\ \eqref{objectivefunctionalWTPS} (i.e.\  maximizing $P_{mech}$ when $P_{mech}<P_{stl}$, and maintaining $P_{mech}$ at the rated level, i.e., $P_{mech}=P_{stl}$, when it is physically  possible to produce $P_{mech} > P_{stl}$). As for the smoothing approximation approach, it seems that the smoothed trajectories converge to the optimal trajectory as the value of $N$ gets larger.
Similarly, in Figure \ref{Smoothing, Gaussian}, we show the applied Gaussian wind profile with $v_0=11 \text{ m/s}$, mean $\mu=20$, and standard deviation $\sigma = 0.5$. Similar to the ramp wind profile case, the optimal trajectory given by the LD-derivatives based approach is effective in steering the system's dynamics as desired, according to objective functional in Eq.\ \eqref{objectivefunctionalWTPS} and the criteria described above in the ramp wind profile case. In this case also, the effectiveness of the smoothed approach depends on how large $N$ is.  \ha{We emphasize that, here and in general, there are typically no theoretical guarantees that a smoothing approximation approach solution will converge to the true solution as the approximation converges to the nonsmooth function. The LD-derivative based solution, on the other hand, directly treats the nonsmoothness without approximations, and hence only faces other unavoidable numerical errors (e.g., rounding error), which is a clear advantage.}

\begin{figure}[!h]\centering
{\includegraphics[width=0.9\columnwidth]{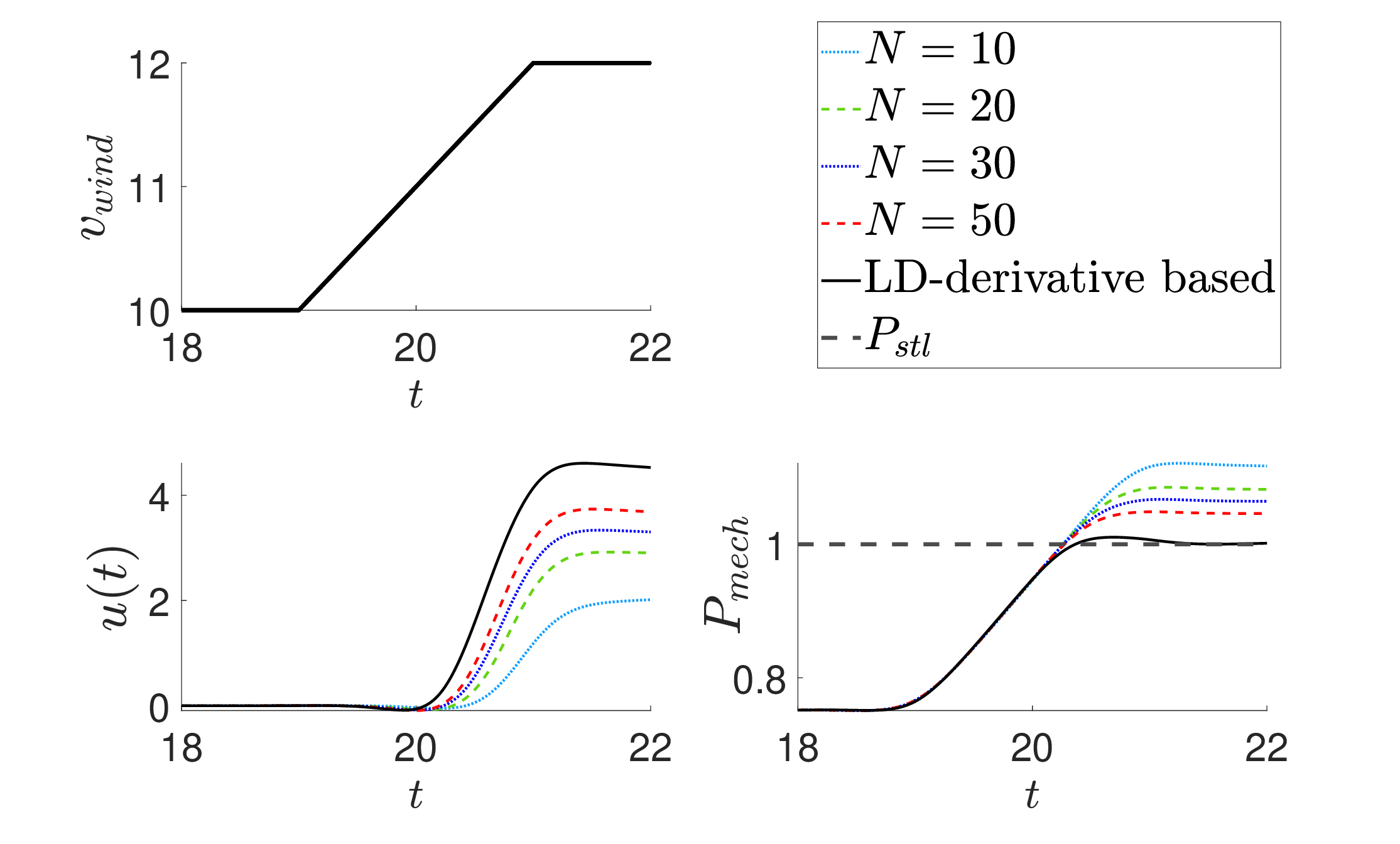}}
\caption{Smoothing approximation with different values of the smoothing parameters $N$ vs. LD-derivatives based sequential approach with a ramp wind profile.
} \label{Smoothing, Ramp}
\end{figure}

\begin{figure}[!h]\centering
{\includegraphics[width=0.9\columnwidth]{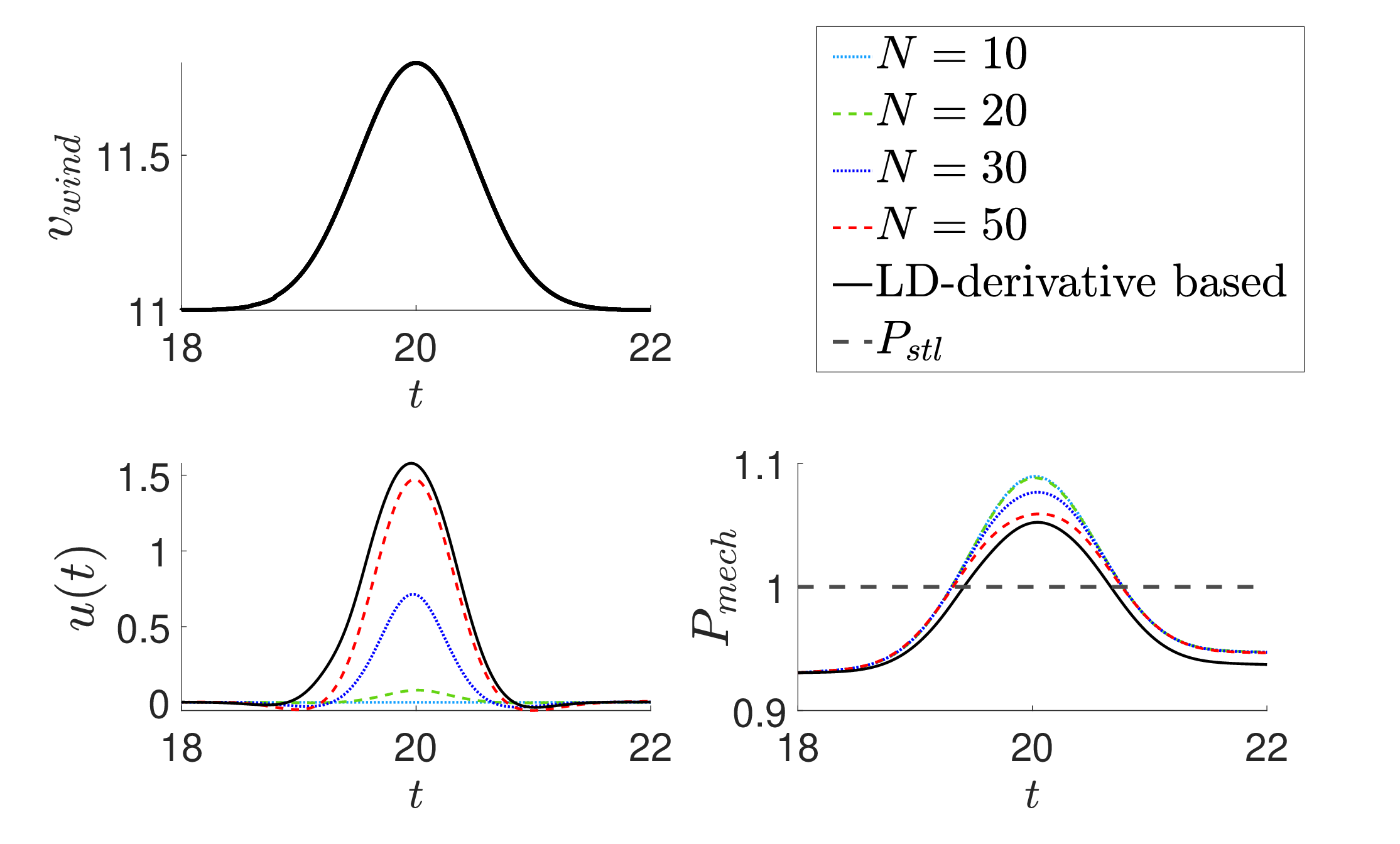}}
\caption{Smoothing approximation with different values of the smoothing parameters $N$ vs. LD-derivatives based sequential approach with a Gaussian wind profile.
\label{Smoothing, Gaussian}}
\end{figure}


\begin{figure}[h!]
    \centering
    \subfloat[]{{\includegraphics[width=0.5\columnwidth]{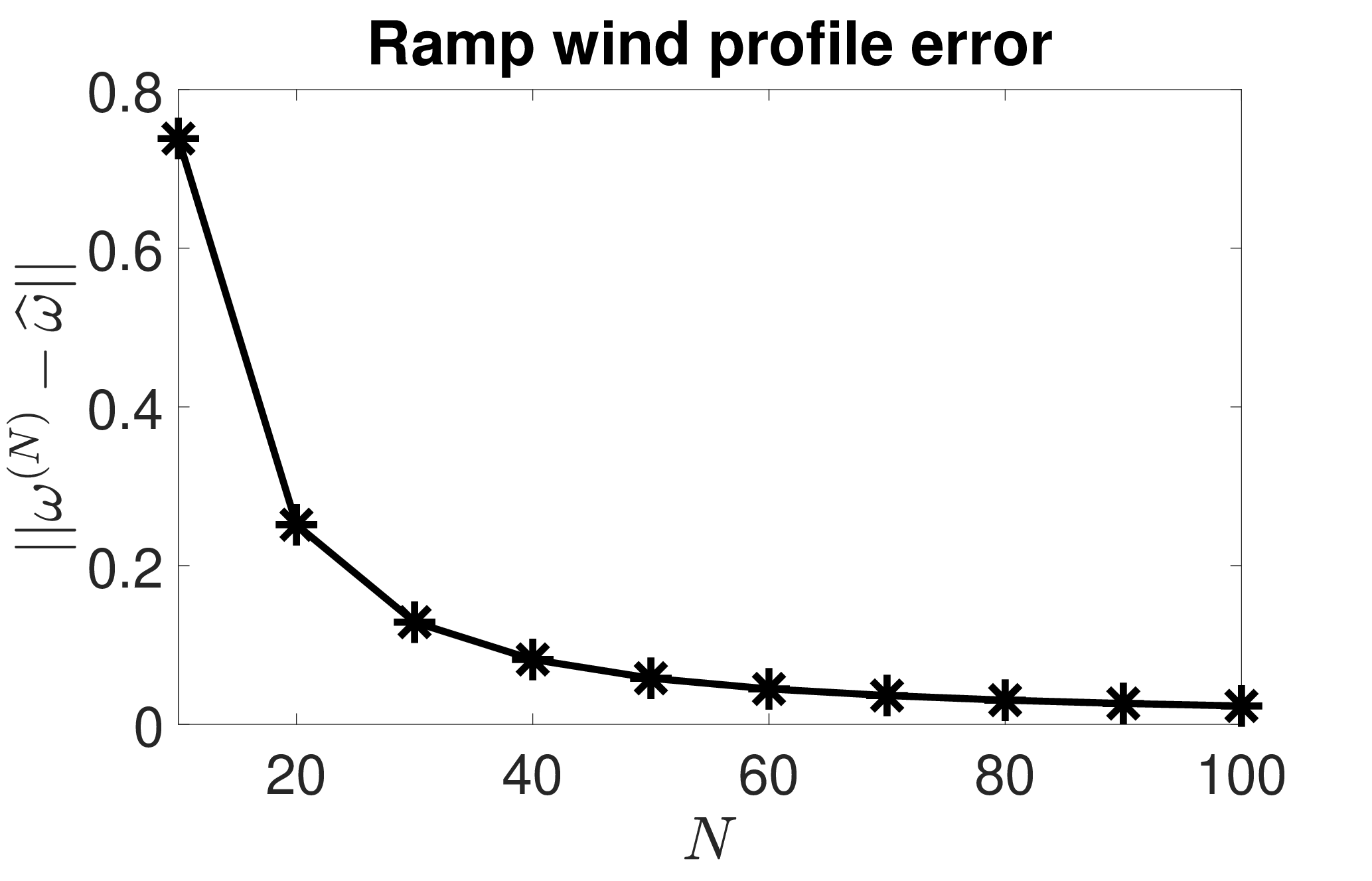}}}%
    \subfloat[\centering ]{{\includegraphics[width=0.5\columnwidth]{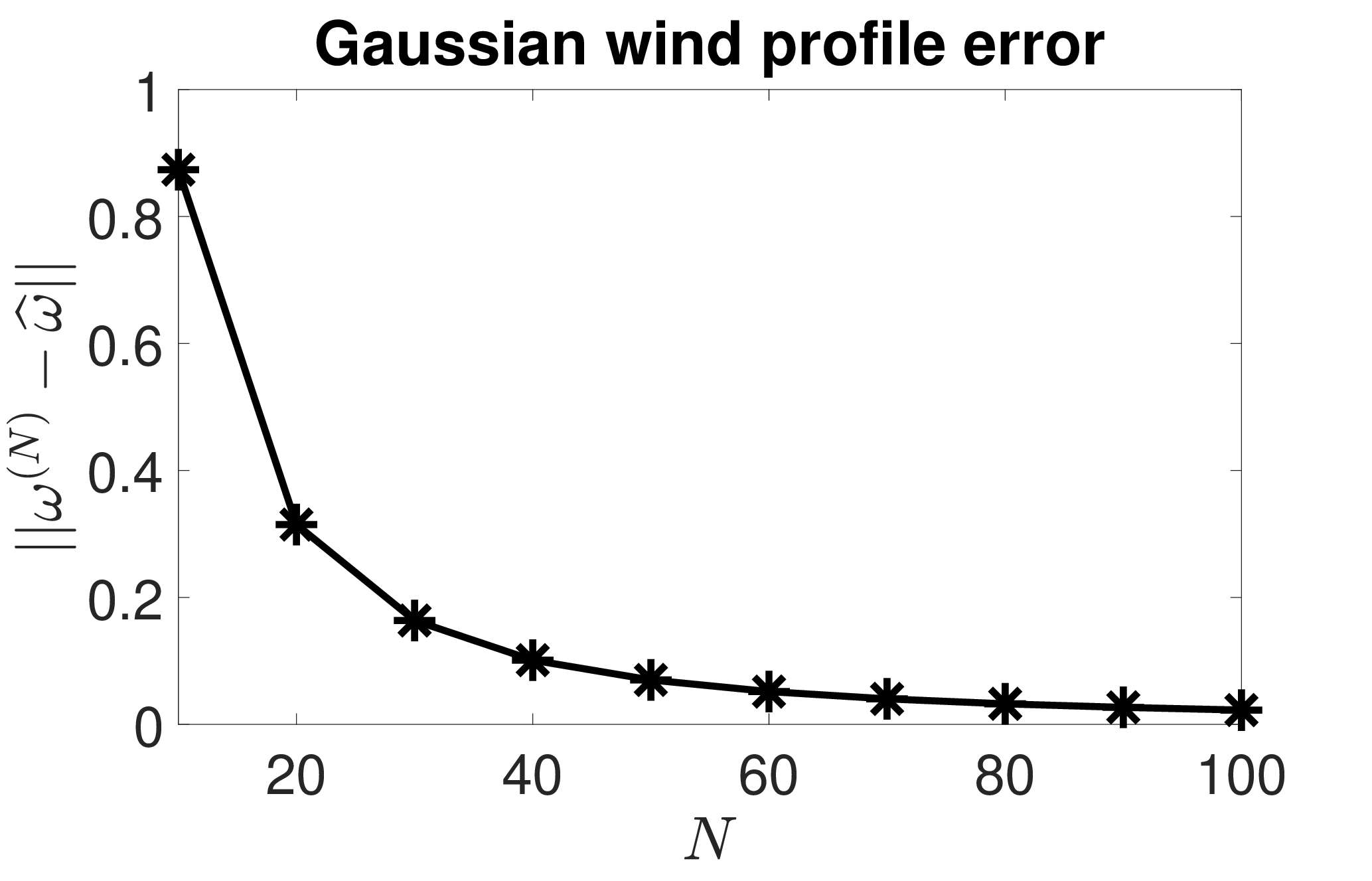} }}%
    \caption{\ha{The 2-norm of error between the nonsmooth function, $\omega$, in \eqref{eq:objfunWPTSintegrand}, and the smoothed  function $\widehat{\omega}^{(N)}$ in \eqref{smoothingobjectivefucntionalWTPS}, evaluated at different values of the smoothing parameters $N$ for a ramp wind profile (left panel) and a Gaussian wind profile (right panel).}} %
     \label{Fig:2normError_Smoothing}
\end{figure}

\ha{A smoothing approach can cause numerical errors that are not traceable. In addition, approximations based on smoothing usually require user-defined parameters that may alter the original problem in an arbitrary manner; in fact, this point was demonstrated in a detailed study by the authors in \cite{eisa2021sensitivity} with relevance to WTPSs.}  
\subsection{Comparison with a Naive Approach}\label{subsection:IVB}

 The second approach is \ha{the direct application of an off-the-shelf smooth optimal control solver to the system} exhibiting the nonsmoothness without any modifications (i.e., the nonsmoothness is ignored).  
In particular,  we compare our approach vs. solving the nonsmooth WTPS optimal control problem using a popular smooth optimal control solver, GPOPS-II \cite{patterson2014gpops},  without providing any (generalized) derivative information. That is, we ignore the presence of nonsmoothness completely (hence, ``naive approach'') and pass the problem to a solver without a smoothing approximation or any other alterations. GPOPS-II  \cite{patterson2014gpops} is \ha{a powerful optimal control solver that is compatible with MATLAB\textsuperscript{\textregistered}. In GPOPS-II, hp-adaptive Legendre-Gauss-Radau quadrature orthogonal collocation method is used so that the optimal control problem is transcribed to a large sparse NLP. GPOPS-II uses the simultaneous collocation method, which is another direct optimal control method that also solves a trajectory optimization problem by transforming it into an NLP problem.
Both the simultaneous collocation method and the sequential method used in this paper are considered as direct methods as opposed to indirect methods which typically yield an auxiliary set of optimality conditions that can be difficult to solve, especially in a nonsmooth setting.
The key difference between the simultaneous collocation method and the sequential method used in this paper is the discretized parameters. In the sequential method, we only discretize the controls, but in the simultaneous collocation method, we discretize both the controls and the states. The reader can refer to \cite{Kelly2017,Rao2010,patterson2014gpops,biegler2007overview} for more details. In order to achieve a specified accuracy, GPOPS-II implements an adaptive mesh refinement method for determining the number of mesh intervals and the degree of the approximating polynomial within each mesh interval.  Sparse finite-differencing of the optimal control problem functions is used to furnish all derivatives required by the NLP solver. In this subsection, we compare the results of our proposed LD-derivatives based sequential optimal control solver with the results of GPOPS-II. 
We note that the reader can refer to this GitHub repository \cite{githubwtps} for the GPOPS-II implementation used to produce the following results.} 

Moving forward to the comparison of results, 
we use the nonsmooth objective functional \eqref{objectivefunctionalWTPS} with a Gaussian wind profile in two different cases: (i) wind speed range that is strictly in Region 2 (medium wind speeds) to avoid any nonsmoothness; and (ii) wind speed range that causes transitions through the nonsmoothness in the objective functional (Regions 2-3). 
In the first case, when there is no nonsmoothness present, the solutions from our approach and the GPOPS-II are nearly identical; this is expected, but it highlights that our approach does not need a priori knowledge of nonsmoothness and recovers smooth solutions automatically. We show the results of the second case in Figure \ref{GPOPS2vsLD, Gaussian, v_0=11}, where the upper panel shows the Gaussian wind profile with mean $\mu=20$, standard deviation $\sigma = 0.5$, with different initial wind speeds. The lower left panel shows the optimal control response, $u(t)$, and the lower right panel shows the effect of the optimal control response on $P_{mech}$. Unlike the first case, we can see from the results of the second case (when nonsmoothness is present) that the solution from GPOPS-II software (which is designed for smooth dynamics/objective functional) gives inaccurate results (breaks) as can be seen in the left panels (with red curves), while the solution from the nonsmooth sequential LD-derivatives based method provides the expected trajectory as seen in the right panels (with black curves). This shows the robustness and generalized nature of the proposed method in all Regions of operation of WTPSs. 

\begin{figure}[ht!]\centering
{\includegraphics[width=0.8\columnwidth]{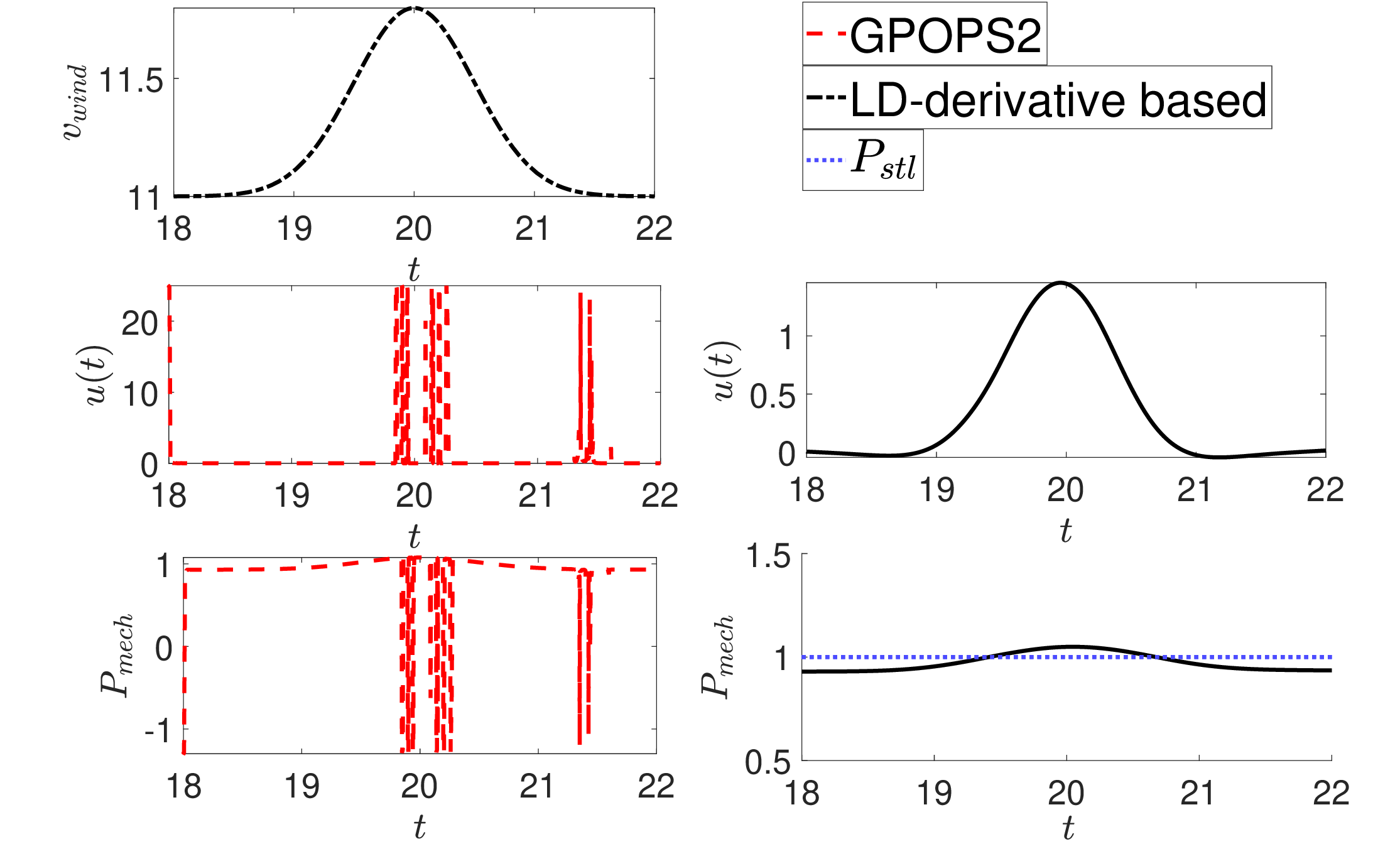}}
\caption{The collocation optimal control solution using GPOPS-II vs. the LD-derivatives based sequential approach. The upper panel shows a Gaussian wind profile with an initial wind speed $v_0=11.0 \text{ m/s}$, and a standard deviation $\sigma = 0.5$. The left panels (red curves) show the optimal control response $u(t)$ and the effect of the control response on $P_{mech}$ for the smooth solver GPOPS-II. The right panels (in black curves) show the optimal control response and the resulting $P_{mech}$ for the LD-derivatives based sequential nonsmooth optimal control method}\label{GPOPS2vsLD, Gaussian, v_0=11}
\end{figure}

  \ha{The ``naive'' approach ignores the nonsmoothness completely and hence comes with no guarantees whatsoever of convergence, accuracy, etc. Unsurprisingly then, a method designed for smooth functions, where conventional numerical differentiation tools (such as finite differencing) are used to provide the derivative information required by the solver, may fail. This is what is observed above in Figure \ref{GPOPS2vsLD, Gaussian, v_0=11}, as the solver breaks in some cases. On the other hand, the main advantage of our proposed method is that it directly treats nonsmooth functions in the system's dynamics and/or the objective functional by providing derivative information via LD-derivatives. And hence, while this approach can face other typical numerical issues (e.g., rounding error), we do not face this same potential for failure as the naive approach.}
  \ha{Lastly, we note that a performance comparison in terms of computational efficiency between  our LD-derivatives based solver and the smooth optimal control solver, GPOPS-II,  is  not appropriate at this stage; the numerical implementation based on our LD-derivatives based approach is a prototype, and its implementation was not focused on optimizing computational efficiencies, while GPOPS-II is a commercial powerful smooth optimal control solver that uses state-of-the-art techniques for solving a wide range of smooth optimal control problems.} 
\section{Conclusions}\label{section:VI}
In this paper, we provided a framework for a nonsmooth sequential optimal control approach to solve the WTPS control problem. \ha{Before summarizing,  we note that
the MATLAB files used to produce the simulations in subsections \ref{subsection:IIIC} and \ref{subsection:IVB}  are uploaded to GitHub as an open-source code \cite{githubwtps}. We note that the GPOPS-II software is a commercial software, so, although our implementation is open, the user will require GPOPS-II prior to running the m-file that uses GPOPS-II.} 
The \ha{key advantages of the proposed approach are summarized as follows:}
\begin{itemize}
    \item The proposed approach incorporates recent significant developments in WTPS modeling and computationally relevant nonsmooth optimal control theory for DAEs.
    \item  The proposed approach is applicable to WTPSs in all wind speed ranges (Regions 1-3) and achieves two targets simultaneously: (i) maximization of power during lower wind speeds; and (ii) maintaining the power at the rated level during higher wind speeds.
    \item We demonstrated the effectiveness of the proposed approach via simulations, including real-world data, and a comparison with alternative approaches, such as smoothing (which is lacking) and a naive approach (which fails).
    \item Having access to the feasible optimal control solution of the WTPS problem extends many possibilities in control developments and, more importantly, provides researchers/industry with a reference solution for comparisons between competing control methods; after all, all control methods should try to realize and be compared with the \textit{optimal control}.
    \item The current work can inspire similar approaches in other power systems in the form of \eqref{NonSmooth_DAE1}-\eqref{NonSmooth_DAE2}, especially renewable energy power systems which can experience nonsmooth conditions and need substantial control developments. 
\end{itemize}   

\appendix

\section{}\label{appendix}
\noindent Denotations of state and parameters in the WTPS model \eqref{WTPS_model}:

\noindent $\rho$, $A_r$, $v_{wind}$ - air density, rotor area $(m^2)$, wind speed (m/s). \\
$C_p$, $w_{ref}$: aerodynamic power coefficient, reference speed.\\
$\alpha_{i,j} \:, P_{mech}$: empirical constants, mechanical power.\\
$w_0 \:, w_{base}$: initialized speed constant, base angular frequency.\\
$P_{elec}$, $V$: electrical (active) power, terminal voltage. \\
$R$, $X$, $E$: infinite bus resistance, reactance, voltage.\\
$Q_{gen}$, $I_{plv}$: reactive power, active current. \\
$\lambda$, $\theta$: tip ratio, pitch angle in degrees. \\
$w_t\:,w_g$: dynamic turbine and generator speeds.\\
$\Delta\theta_m$: integral of difference between $w_t$ and $w_g$.\\
$f_1$: integral of differences of speeds.\\
$P_{stl}$: rated power.\\
$P_{inp}$, $P_{1elec}$: power order, filtered electrical power. \\
$T_{pc}$, $K_{ptrq}$: time constant, torque control proportional.\\
$V_{ref}$, $E_{qcmd}$: reference voltage, reactive voltage command.\\
$H$, $H_g$: turbine inertia constant, generator inertia constant.\\
$D_{tg}$, $K_{tg}$ - shaft damping constant, shaft stiffness constant.\\
$K_{itrq}$: torque control constant.\\
$K_{Qi}$: reference voltage gain.\\
$T_{pwr}$, $K_{vi}$, $E_q$: filtered electric power time constant, reactive voltage command time constant, generator reactive variable.\\

\noindent The values of the parameters involved, and the $C_p$ coefficients are given in Tables
\ref{tab:params}, and \ref{tab:cpcoeftable}, respectively \cite{eisa2019modeling,eisa2019nonlinear,IJDY17}.
\begin{table}[!h]
\caption{Representative values for parameters in the model}
\label{tab:params}
\begin{center}
\begin{tabular}{|c||c|}
\hline
Parameter & Value \\
\hline
$w_0$ & 1 (any choice bigger than 0)\\
\hline
$X_{eq},D_{tg} $ & 0.8,1.5\\
\hline
$K_{tg},w_{base}$  & 1.11,125.66  respectively \\
\hline
$\frac{1}{2}\rho A_r, K_b$  & 0.00159 and 56.6 respectively \\
\hline
$H$  &  4.94 \\
\hline
$H_g, K_{itrq}$  & 0.62 ,0.6 respectively\\
\hline
$T_{pc},K_{ptrq}$  & 0.05, 3 respectively \\
\hline
$T_{pwr},K_{Qi}$  & 0.05, 0.1 respectively\\
\hline
$P_{stl},w_{ref}$  & 1 , 1.2 respectively\\
\hline
$K_{vi}, R, E $  & 40, 0.02, 1.0164 respectively\\
\hline
$X=X_l+X_{tr} $  & $X_l=0.0243, X_{tr}=0.00557 $\\
\hline
\end{tabular}
\end{center}
\end{table}
\begin{table}[!ht]
\caption{$c_p$ coefficients $\alpha_{i,j}$}
\label{tab:cpcoeftable}
\begin{center}
\begin{tabular}{|c||c||c||c||c||c|}
\hline
i & j & $\alpha_{i,j}$ & i & j & $\alpha_{i,j}$ \\
\hline
4 & 4 & 4.9686e-10 & 4 & 3 & -7.1535e-8 \\
\hline
4 & 2 & 1.6167e-6 & 4 & 1 & -9.4839e-6 \\
\hline
4 & 0 & 1.4787e-5 & 3 & 4 & -8.9194e-8 \\
\hline
3 & 3 & 5.9924e-6 & 3 & 2 & -1.0479e-4 \\
\hline
3 & 1 & 5.7051e-4 & 3 & 0 & -8.6018e-4 \\
\hline
2 & 4 & 2.7937e-6 & 2 & 3 & -1.4855e-4 \\
\hline
2 & 2 & 2.1495e-3 & 2 & 1 & -1.0996e-2 \\
\hline
2 & 0 & 1.5727e-2 &-&-&-\\
\hline
1 & 4 & -2.3895e-5 & 1 & 3 & 1.0683e-3 \\
\hline
1 & 2 & -1.3934e-2 & 1 & 1 & 6.0405e-2 \\
\hline
1 & 0 & -6.7606e-2 & 0 & 4 & 1.1524e-5 \\
\hline
0 & 3 & -1.3365e-4 & 0 & 2 & -1.2406e-2 \\
\hline
0 & 1 & 2.1808e-1 & 0 & 0 & -4.1909e-1 \\
\hline
\end{tabular}
\end{center}
\end{table}
\clearpage

\bibliographystyle{elsarticle-num} 





\end{document}

%% file: PSmacros.tex
 \usepackage{etex}
 
\usepackage{lipsum}
\usepackage{amsfonts}
\usepackage{graphicx}
\usepackage{epstopdf} 
\usepackage{enumerate} 
 \usepackage{amsmath}  
 \usepackage{subfig,bbm,enumerate,multirow,framed,blkarray}

\usepackage{verbatim,tcolorbox}
  \usepackage{tikz}
\usetikzlibrary{shapes,arrows}
\usetikzlibrary{positioning}  
\usepackage{pstricks-add} 
 \usepackage{wrapfig}
 \usepackage{pgfplots}
  
\usepackage{color}
\usepackage{savesym}
 \usepackage{algpseudocode} 

 \usepackage{graphicx}          

\newtheorem{theorem}{Theorem}[section]

\newtheorem{remark}[theorem]{Remark}
 
\newtheorem{example}[theorem]{Example}

\newcommand{\conv}{\text{conv}}

	\algnewcommand\And{\textbf{and }}
	\algnewcommand\Or{\textbf{or }}

\newcommand{\xnot}{\x_0}

\newcommand{\mumu}{\pmb{\mu}}

\newcommand{\lshiftmin}{\mathbf{slmin}}

\newcommand{\coord}{\mathbf{e}}

\newcommand{\real}{\mathbb R}
\newcommand{\posint}{\mathbb{N}}

\newcommand{\etaeta}{\pmb{\eta}}

\newcommand{\uu}{\mathbf{u}}
\newcommand{\vv}{\mathbf{v}}

\newcommand{\f}{\mathbf{f}}

\newcommand{\F}{\mathbf{F}}

\newcommand{\g}{\mathbf{g}}

\newcommand{\h}{\mathbf{h}}
\newcommand{\x}{\mathbf{x}}

\newcommand{\y}{\mathbf{y}}
\newcommand{\M}{\mathbf{M}}

\newcommand{\II}{\mathbf{I}}

\newcommand{\capy}{\mathbf{Y}}

\newcommand{\pp}{\mathbf{p}} 

\newcommand{\zero}{\mathbf{0}}

\newcommand{\Jf}{\mathbf{Jf}}
\newcommand{\J}{\mathbf{J}}

\algnewcommand{\algorithmicgoto}{\textbf{go to}}%
\algnewcommand{\Goto}[1]{\algorithmicgoto~\ref{#1}}%
 
\newenvironment{bsmallmatrix}{\left[\begin{smallmatrix}}{\end{smallmatrix}\right]}